\documentclass{amsart}%
\usepackage{pgf,tikz}
\usepackage{amssymb}
\usepackage{amsmath}
\usepackage{amsfonts}
\usepackage{graphicx}
\usepackage{bm}
\usepackage{pgfplots}%
\providecommand{\U}[1]{\protect\rule{.1in}{.1in}}
\usetikzlibrary{backgrounds}
\usetikzlibrary{arrows}
\pgfplotsset{compat=1.15}
\let\oldmathbf\mathbf
\renewcommand{\mathbf}[1]{\boldsymbol{\oldmathbf{#1}}}
\newtheorem{theorem}{Theorem}
\newtheorem*{theorem*}{Theorem}

\newtheorem{definition}[theorem]{Definition}

\newtheorem{lemma}[theorem]{Lemma}

\newtheorem{proposition}[theorem]{Proposition}
\newtheorem{remark}[theorem]{Remark}

\allowdisplaybreaks
\definecolor{zzttqq}{rgb}{0.6,0.2,0.}
\begin{document}
\title{Irregularities of distribution and geometry of planar convex sets}
\author[L. Brandolini]{Luca Brandolini}
\address{Dipartimento di Ingegneria Gestionale, dell'Informazione e della Produzione,
Universit\`{a} degli Studi di Bergamo, Viale Marconi 5, 24044 Dalmine BG, Italy}
\email{luca.brandolini@unibg.it}
\author[G. Travaglini]{Giancarlo Travaglini}
\address{Dipartimento di Matematica e Applicazioni, Universit\`{a} di Milano-Bicocca,
Via Cozzi 55, 20025 Milano, Italy}
\email{giancarlo.travaglini@unimib.it}

\begin{abstract}
We consider a planar convex body $C$ and we prove several analogs of Roth's
theorem on irregularities of distribution. When $\partial C$ is $\mathcal{C}%
^{2}$ regardless of curvature, we prove that for every set $\mathcal{P}_{N}$
of $N$ points in $\mathbb{T}^{2}$ we have the sharp bound%
\[
\int_{0}^{1}\int_{\mathbb{T}^{2}}\left\vert \mathrm{card}\left(
\mathcal{P}_{N}\mathcal{\cap}\left(  \lambda C+t\right)  \right)  -\lambda
^{2}N\left\vert C\right\vert \right\vert ^{2}~dtd\lambda\geqslant cN^{1/2}\;.
\]
When $\partial C$ is only piecewise $\mathcal{C}^{2}$ and is not a polygon we
prove the sharp bound%
\[
\int_{0}^{1}\int_{\mathbb{T}^{2}}\left\vert \mathrm{card}\left(
\mathcal{P}_{N}\mathcal{\cap}\left(  \lambda C+t\right)  \right)  -\lambda
^{2}N\left\vert C\right\vert \right\vert ^{2}~dtd\lambda\geqslant cN^{2/5}.
\]
We also give a whole range of intermediate sharp results between $N^{2/5}$ and
$N^{1/2}$. Our proofs depend on a lemma of Cassels-Montgomery, on ad hoc
constructions of finite point sets, and on a geometric type estimate for the
average decay of the Fourier transform of the characteristic function of $C$.

\end{abstract}
\maketitle

\section{Introduction}

The term \textit{Irregularities of distribution}, often replaced with
(\textit{geometric})\textit{\ discrepancy}, has been introduced by K. Roth in
his seminal paper\footnote{K. Roth is famous for many results, first of all
his solution of the Siegel conjecture concerning approximation of algebraic
numbers by rationals. It is known that he considered \cite{roth} to be his
best work (see \cite{CV}).} \textit{On irregularities of distribution}
(published in Mathematika in 1954, \cite{roth}), where the following result
has been proved.

\begin{theorem}
Let $N>1$ be an integer, and let $u_{1},u_{2},\ldots,u_{N}$ be $N$ points, not
necessarily distinct, in the square $\left[  0,1\right)  ^{2}$. Then%
\begin{equation}
\int_{0}^{1}\int_{0}^{1}\left(  S\left(  x,y\right)  -Nxy\right)
^{2}\ dxdy>c\ \log\left(  N\right)  \ , \label{roth}%
\end{equation}
where
\begin{equation}
S\left(  x,y\right)  =\mathrm{card}\left\{  j=1,2,\ldots,N:u_{j}\in\left[
0,x\right)  \times\left[  0,y\right)  \right\}  \ . \label{Sxy}%
\end{equation}

\end{theorem}

From here on out, $c,\ c_{1},\ldots\ $ are positive constants, independent of
$N$, which may change from step to step.

Given any finite point set $\mathcal{P}_{N}=\left\{  u_{j}\right\}  _{j=1}%
^{N}\subset\left[  0,1\right)  ^{2}$, Roth's theorem concerns the $L^{2}$
discrepancy between ($N$ times) the area of the rectangle $\left[  0,x\right)
\times\left[  0,y\right)  $ and the number of points $u_{j}$ that belong to
the above rectangle. More generally, Discrepancy Theory concerns the problem
of replacing a continuous object with a discrete sampling, and is presently a
crossroads between many fields of Mathematics (see e.g. \cite{BC}, \cite{BDP},
\cite{BGT}, \cite{chazelle}, \cite{CST}, \cite{dick}, \cite{DT},
\cite{matousek}, \cite{travaglini}).

\medskip

Roth's paper dealt with the van der Corput conjecture\footnote{If $s_{1}%
,s_{2},s_{3},\ldots$ is an infinite sequence of real numbers lying between $0$
and $1$, then corresponding to any arbitrarily large $k$, there exist a
positive integer $n$ and two subintervals, of equal length, of $\left(
0,1\right)  $, such that the number of $s_{v}$ ($v=1,\ldots,n$) that lie in
one of the subintervals differs from the numbers of such $s_{v}$ that lie in
the other subinterval by more than $k$.}, that is a $1$-dimensional problem
about infinite numerical sequences that turns into a $2$-dimensional geometric
problem about distributions of finite point sets with respect to a family of
rectangles. Roth not only improved the quantitative solution of the van der
Corput conjecture previously obtained by T. van Aardenne-Ehrenfest
(\cite{Vana45}, \cite{Vana49}), but he also introduced a geometric point of
view and \textquotedblleft started a new field\textquotedblright. \medskip

In 1956 H. Davenport \cite{davenport} proved that the $\log\left(  N\right)  $
lower bound in (\ref{roth})\ cannot be improved. He showed that, for every
$N$, there exist $N$ points $u_{1},u_{2},\ldots,u_{N}$ in the square $\left[
0,1\right)  ^{2}$ such that%
\[
\int_{0}^{1}\int_{0}^{1}\left(  S\left(  x,y\right)  -Nxy\right)
^{2}\ dxdy\leqslant c\log\left(  N\right)  ~,
\]
where $S\left(  x,y\right)  $ is as in (\ref{Sxy}).\medskip

H. Montgomery \cite[Ch. 6]{montgomery} introduced a different point of view
and used Fourier series to prove the following result.

\begin{theorem}
\label{square}For every finite set $\mathcal{P}_{N}$ of $N$ points in
$\mathbb{T}^{2}=\mathbb{R}^{2}/\mathbb{Z}^{2}$ we have%
\begin{equation}
\int_{0}^{1}\int_{\mathbb{T}^{2}}\left\vert \mathrm{card}\left(
\mathcal{P}_{N}\mathcal{\cap}\left(  \left[  0,\lambda\right)  ^{2}+t\right)
\right)  -\lambda^{2}N\right\vert ^{2}\ dtd\lambda\geqslant c\log\left(
N\right)  \ .\label{montgomery}%
\end{equation}

\end{theorem}

M. Drmota proved that the LHS in (\ref{roth}) and the LHS in (\ref{montgomery}%
) are equivalent (see \cite{drmota1}, see also \cite{ruzsa}).\medskip

It is natural to replace the rectangle $\left[  0,x\right)  \times\left[
0,y\right)  $ in Roth's theorem with other geometric objects, first of all
suitable families of convex bodies (that is, bounded convex sets with
non-empty interiors). Then the lower bound of the discrepancy may be much
larger than a logarithm, as W. Schmidt first pointed out considering the case
of a ball (see \cite{schmidt}). More generally, we can consider an arbitrary
convex body $C$ and average its discrepancy over translations, dilations and
rotations. J. Beck \cite{beck} and H. Montgomery (see \cite[Ch. 6]%
{montgomery}) proved independently the following result (which we state only
in the planar case).

\begin{theorem}
\label{BM}Let $C\subset\mathbb{T}^{2}$ be a convex body of diameter less than
$1$. Then for every set $\mathcal{P}_{N}$ of $N$ points in $\mathbb{T}^{2}$ we
have%
\begin{equation}
\int_{0}^{1}\int_{SO\left(  2\right)  }\int_{\mathbb{T}^{2}}\left\vert
\mathrm{card}\left(  \mathcal{P}_{N}\mathcal{\cap}\left(  \lambda\sigma\left(
C\right)  +t\right)  \right)  -\lambda^{2}N\left\vert C\right\vert \right\vert
^{2}~dtd\sigma d\lambda\geqslant cN^{1/2}\;, \label{bm}%
\end{equation}
where $\left\vert C\right\vert $ denotes the area of $C$.
\end{theorem}

The lower bound in (\ref{bm}) is sharp for every convex body. This follows
from a classical result of D. Kendall \cite{kendall} on lattice points,
together with a Fourier analytic result proved in\ \cite{P1} (see also
\cite[Ch. 8]{travaglini} and \cite{BHI}).

The integration over dilations in (\ref{bm}) cannot be avoided (see \cite{TT},
see also \cite{BCGT}, \cite{PS}, \cite[Ch. 11]{travaglini} for results in
higher dimensions). The above result of Davenport (see also \cite{CT}) shows
that also the integration over rotations is necessary in (\ref{bm}).

If we replace $C$ with a disk $D$ in Theorem \ref{BM}, then the integration
over rotations is meaningless and (\ref{bm}) reduces to the following (sharp)
inequality.
\[
\int_{0}^{1}\int_{\mathbb{T}^{2}}\left\vert \mathrm{card}\left(
\mathcal{P}_{N}\mathcal{\cap}\left(  \lambda D+t\right)  \right)  -\lambda
^{2}N\left\vert D\right\vert \right\vert ^{2}~dtd\lambda\geqslant cN^{1/2}\;.
\]

In short, after averaging the discrepancy over translations and dilations, we
have $\log\left(  N\right)  $ as a (sharp) lower estimate for the case of the
square and $N^{1/2}$ as a (sharp) lower estimate for the case of the
disk.\medskip

From now on we will always average the discrepancies over translations and
dilations. \medskip

The above $\log\left(  N\right)  $ lower estimate has been extended from the
case of a square to the case of convex polygons in \cite{drmota1}, while the
$N^{1/2}$ lower estimate has been extended from the case of a disk to the case
of convex bodies with $\mathcal{C}^{12}$ boundary having everywhere positive
curvature in \cite{drmota0} (see also \cite[Ch. 7]{BC}). In \cite{BT} we
constructed an example of a convex body $C$ whose boundary is $\mathcal{C}%
^{2}$ with a flat point, such that the $N^{1/2}$ lower bound still holds true.

Related results have been proved by J. Beck in \cite{beck2}, where he obtained
lower bounds for the discrepancy in terms of the quality of approximation of
$C$ through inscribed polygons. A related point of view has been considered in
\cite{BIT}.\bigskip

We prove that every convex body with $\mathcal{C}^{2}$ boundary, regardless of
curvature, satisfies the sharp $N^{1/2}$ lower bound. We also prove that if we
are not in the case of a polygon and the boundary is piecewise $\mathcal{C}%
^{2}$, then the lower bound $N^{2/5}$ holds true. Moreover, for every
$2/5\leqslant a\leqslant1/2$ we give a geometric condition that implies the
lower bound $N^{a}$. We also give an explicit convex body which admits the
upper bound $N^{a}$ for the discrepancy.\bigskip

\section{Main results}

Here we state the main results of the paper. The proofs are in the next
sections.\medskip

The following Theorem \ref{C2} is a particular case of both Theorem
\ref{Theorem B-a} and Theorem \ref{irredistr} below.\medskip

\begin{theorem}
\label{C2}Let $C\subset\mathbb{T}^{2}$ be a convex body with $\mathcal{C}^{2}$
boundary. Then there exists $c>0$ such that for every set $\mathcal{P}_{N}$ of
$N$\ points in $\mathbb{T}^{2}$ we have%
\begin{equation}
\int_{0}^{1}\int_{\mathbb{T}^{2}}\left\vert \mathrm{card}\left(
\mathcal{P}_{N}\cap\left(  \lambda C+t\right)  \right)  -\lambda
^{2}N\left\vert C\right\vert \right\vert ^{2}~dtd\lambda\geqslant
cN^{1/2}.\label{uno}%
\end{equation}

\end{theorem}

The case of a disk shows that (\ref{uno}) is sharp.

A few definitions and preliminary results are necessary before we state our
next theorems.\medskip

We are going to consider convex planar bodies which are not polygons and have
piecewise $\mathcal{C}^{2}$ boundaries. For the discrepancy associated to
these bodies we prove sharp lower bounds depending on the regularity of the
boundary, which we measure in terms of lengths of chords. In some cases this
condition can be interpreted as an \textquotedblleft inner disk
condition\textquotedblright. See Remark \ref{remarkdfdf}.

\begin{definition}
\label{firstdef} Let $C\subset\mathbb{T}^{2}$ be a convex body. For every unit
vector $\Theta=\left(  \cos\theta,\sin\theta\right)  $ and $\delta>0$ we
consider the chord%
\[
\gamma_{\Theta}\left(  \delta\right)  =\left\{  x\in C:x\cdot\Theta=\inf_{y\in
C}\left(  y\cdot\Theta\right)  +\delta\right\}
\]
and its length $\left\vert \gamma_{\Theta}\left(  \delta\right)  \right\vert
$. See Figure \ref{chord7}. \begin{figure}[h]
\begin{center}
\includegraphics[width=60mm]{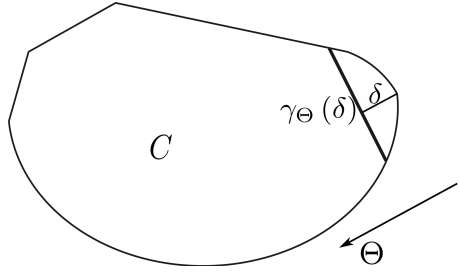}
\end{center}
\caption{The chord $\gamma_{\Theta}\left(  \delta\right)  $. }%
\label{chord7}%
\end{figure}
\end{definition}

Observe that $\left\vert \gamma_{\Theta}\left(  \delta\right)  \right\vert $
cannot be too small. Namely we have the following result.

\begin{proposition}
\label{stimadeltadasotto} Let $C$ be a planar convex body. Then there exist
$\delta_{0},c>0$ such that for $0<\delta<\delta_{0}$ and every direction
$\Theta$ we have%
\[
\left\vert \gamma_{\Theta}\left(  \delta\right)  \right\vert \geqslant
c\delta.
\]

\end{proposition}

Observe that if $C$ is a polygon, then the above bound is sharp for all but a
finite number of directions. If $\partial C$ is smooth enough, then
$\left\vert \gamma_{\Theta}\left(  \delta\right)  \right\vert $ can be much
larger. See Remark \ref{CordaC2} and Theorem \ref{Equivalenza} below.

\begin{theorem}
\label{Theorem B-a}Let $C$ be a convex body and let $\Theta=\left(  \cos
\theta,\sin\theta\right)  $. Assume the existence of constants $\delta
_{0},c_{1},c_{2}>0$, $1/2\leqslant\sigma\leqslant1$ and an interval $I$ in
$\left(  -\pi,\pi\right)  $ such that for every $0<\delta\leqslant\delta_{0}$
we have%
\begin{equation}
\left\{
\begin{array}
[c]{ll}%
\left\vert \gamma_{-\Theta}\left(  \delta\right)  \right\vert +\left\vert
\gamma_{\Theta}\left(  \delta\right)  \right\vert \geqslant c_{1}\delta^{1/2}
& \ \text{for every }\theta\in I,\\[3mm]%
\left\vert \gamma_{-\Theta}\left(  \delta\right)  \right\vert +\left\vert
\gamma_{\Theta}\left(  \delta\right)  \right\vert \geqslant c_{2}%
\delta^{\sigma} & \ \text{for every }\theta\notin I.
\end{array}
\right.  \label{due}%
\end{equation}
Then there exists $c>0$ such that for every set $\mathcal{P}_{N}$ of $N$
points in $\mathbb{T}^{2}$ we have
\begin{equation}
\int_{0}^{1}\int_{\mathbb{T}^{2}}\left\vert \mathrm{card}\left(
\mathcal{P}_{N}\cap\left(  \lambda C+t\right)  \right)  -\lambda
^{2}N\left\vert C\right\vert \right\vert ^{2}~dtd\lambda\geqslant
cN^{2/\left(  2\sigma+3\right)  }\ . \label{sigma_basso}%
\end{equation}

\end{theorem}

The following result shows that Theorem \ref{Theorem B-a} is sharp.

\begin{theorem}
\label{Theorem B-b}For every $1/2\leqslant\sigma\leqslant1$ there exist $c>0$,
an explicit construction of a planar convex body $C_{\sigma}$ that satisfies
(\ref{due}), and finite sets $\mathcal{P}_{N_{j}}\subset\mathbb{T}^{2}$ of
cardinality $N_{j}\rightarrow+\infty$ such that%
\begin{equation}
\int_{0}^{1}\int_{\mathbb{T}^{2}}\left\vert \mathrm{card}\left(
\mathcal{P}_{N_{j}}\mathcal{\cap}\left(  \lambda C+t\right)  \right)
-\lambda^{2}N_{j}\left\vert C\right\vert \right\vert ^{2}~dtd\lambda\leqslant
cN_{j}^{2/\left(  2\sigma+3\right)  }\ . \label{sigma_alto}%
\end{equation}

\end{theorem}

\begin{remark}
\label{CordaC2}It is not difficult to see that if $\partial C$ is
$\mathcal{C}^{2}$ then (\ref{due}) is true with $\sigma=1/2$ (therefore
Theorem \ref{C2} is a consequence of Theorem \ref{Theorem B-a}). Indeed, a
suitable choice of coordinates allows us to assume that the origin belongs to
$\partial C$ and that $\left(  0,1\right)  $ is the inward unit normal at the
origin. Hence $\partial C$ coincides locally with the graph of a
$\mathcal{C}^{2}$ function $\varphi\left(  x\right)  $ which satisfies
$0\leq\varphi^{\prime\prime}\left(  x\right)  \leqslant c$, where $c$ depends
only on $C$. A repeated integration of this inequality yields $cx^{2}%
-\varphi\left(  x\right)  \geqslant0$, which in turns gives (\ref{due}) with
$\sigma=1/2$. The above argument can be repeated for every direction. Then
(\ref{due}) holds uniformly for $\sigma=1/2$ and every $\Theta$.
\end{remark}

Observe that the exponent $2/\left(  2\sigma+3\right)  $ in (\ref{sigma_basso}%
) and (\ref{sigma_alto}) takes all values between $2/5$ and $1/2$. The next
proposition is a corollary of Proposition \ref{stimadeltadasotto} and Theorem
\ref{Theorem B-a} and shows that $N^{2/5}$ is a general lower bound for convex
planar bodies that are not polygons and have piecewise $\mathcal{C}^{2}$ boundary.

\begin{proposition}
\label{Prop2/5}Let $C$ be a convex planar body which is not a polygon and has
piecewise $\mathcal{C}^{2}$ boundary. Then there exists $c>0$ such that for
every set $\mathcal{P}_{N}$ of $N$ points in $\mathbb{T}^{2}$ we have
\[
\int_{0}^{1}\int_{\mathbb{T}^{2}}\left\vert \mathrm{card}\left(
\mathcal{P}_{N}\mathcal{\cap}\left(  \lambda C+t\right)  \right)  -\lambda
^{2}N\left\vert C\right\vert \right\vert ^{2}~dtd\lambda\geqslant cN^{2/5}\ .
\]

\end{proposition}

Observe that Theorem \ref{Theorem B-b} shows that the above estimate $N^{2/5}$
cannot be improved.\medskip

Theorem \ref{C2} is a consequence also of the next theorem.

\begin{theorem}
\label{irredistr}Let $C\subset\mathbb{T}^{2}$ be a convex body. Assume there
exist $c,\delta_{0}>0$ and $1/2\leqslant\sigma<1$ such that for $0\leqslant
\delta\leqslant\delta_{0}$ and $\theta\in\left[  0,\pi\right)  $ we have%
\begin{equation}
\left\vert \gamma_{-\Theta}\left(  \delta\right)  \right\vert +\left\vert
\gamma_{\Theta}\left(  \delta\right)  \right\vert \geqslant c\delta^{\sigma}.
\label{DueCorde}%
\end{equation}
Then, there exists $c>0$ such that for every finite set $\mathcal{P}_{N}$ of
$N$ points in $\mathbb{T}^{2}$ we have%
\begin{equation}
\int_{1/2}^{1}\int_{\mathbb{T}^{2}}\left\vert \mathrm{card}\left(
\mathcal{P}_{N}\mathcal{\cap}\left(  \tau C+t\right)  \right)  -\tau
^{2}N\left\vert C\right\vert \right\vert ^{2}~dtd\tau\geqslant c\ N^{1-\sigma
}\ . \label{nunquarto}%
\end{equation}

\end{theorem}

\begin{remark}
We show that there exist planar convex bodies that satisfy the assumptions of
Theorem \ref{irredistr}, but do not satsfy the assumptions of Theorem
\ref{Theorem B-a}. Let $0<\alpha<1$, let $\left\{  q_{r}\right\}
_{r=1}^{+\infty}=\mathbb{Q}\cap\left[  0,1\right]  $, and let%
\[
f\left(  x\right)  =\sum_{r=1}^{+\infty}\frac{1}{r^{2}}\left(  x-q_{r}\right)
_{+}^{\alpha+1},
\]
with%
\[
x_{+}^{\alpha+1}=\left\{
\begin{array}
[c]{ll}%
x^{\alpha+1} & x\geqslant0,\\
0 & x<0.
\end{array}
\right.
\]
Since $x_{+}^{\alpha+1}\in C^{1,\alpha}$ (the space of functions with
H\"{o}lder continuous derivative of order $\alpha$) we immediately have $f\in
C^{1,\alpha}\left(  \left[  0,1\right]  \right)  $. Let us show that
$f\not \in C^{1,\beta}\left(  \left[  0,1\right]  \right)  $ for any
$\beta>\alpha$. Indeed, since $x_{+}^{\alpha}$ is increasing, for every fixed
$r_{0}\in\mathbb{Q}\cap\left[  0,1\right]  $ we have%
\begin{align*}
\sup_{h>0}\frac{f^{\prime}\left(  q_{r_{0}}+h\right)  -f^{\prime}\left(
q_{r_{0}}\right)  }{h^{\beta}} &  =\sup_{h>0}%
{\displaystyle\sum\limits_{r=1}^{+\infty}}
\frac{1}{r^{2}}\left(  \alpha+1\right)  \frac{\left(  q_{r_{0}}+h-q_{r}%
\right)  _{+}^{\alpha}-\left(  q_{r_{0}}-q_{r}\right)  _{+}^{\alpha}}%
{h^{\beta}}\\
&  \geqslant\sup_{h>0}\frac{1}{r_{0}^{2}}\left(  \alpha+1\right)  \frac
{h_{+}^{\alpha}}{h^{\beta}}=+\infty.
\end{align*}
Since $f$ is convex (note that $x_{+}^{\alpha+1}$ is convex) we can construct
a convex body $C$ such that $\partial C$ is $C^{1,\alpha}$ but not
$C^{1,\beta}$ for any $\beta>\alpha$. Theorem \ref{Equivalenza} below yields
our claim.
\end{remark}

If $\partial C$ is $\mathcal{C}^{2}$, then Remark \ref{CordaC2} yields
$\left\vert \gamma_{\Theta}\left(  \delta\right)  \right\vert \geqslant
c\delta^{1/2}$. One may be tempted to say that the inequality $\left\vert
\gamma_{\Theta}\left(  \delta\right)  \right\vert \geqslant c\delta^{1/2}$
implies a suitable regularity on $\partial C$. The following example shows
that this is not always true. Consider a planar convex body $C$ with
$\mathcal{C}^{1}$ boundary $\partial C\ni\left(  0,0\right)  $, where the
inward unit normal is $\Theta_{0}=\left(  0,1\right)  $. Also assume that
$\partial C$ coincides locally with the graph of the function%
\[
\varphi\left(  x\right)  =\left\{
\begin{array}
[c]{ll}%
\left\vert x\right\vert ^{3/2} & \text{if }-\varepsilon<x\leqslant0,\\
x^{2} & \text{if }0<x<\varepsilon.
\end{array}
\right.
\]
Then $\left\vert \gamma_{\Theta_{0}}\left(  \delta\right)  \right\vert
\approx\delta^{1/2}$, but $\partial C$ is not $\mathcal{C}^{2}$ at the origin.
Indeed, $\left\vert \gamma_{\Theta_{0}}\left(  \delta\right)  \right\vert $ is
the sum of two contributions of different order, one coming from $\left\vert
x\right\vert ^{3/2}$ for $x<0$ and the other coming from $x^{2}$ for $x>0$. To
obtain information on the regularity of $\partial C$ one has to consider these
two contributions separately. This is the motivation of the following definition.

\begin{definition}
\label{Def Corda}Let $C$ be a convex planar body where $\partial C$ is
$\mathcal{C}^{1}$. For every unit vector $\Theta$ the chord $\gamma_{\Theta
}\left(  \delta\right)  $ is parallel to the tangent line at a point
$P\in\partial C$ such that
\[
P\cdot\Theta=\inf_{y\in C}\left(  y\cdot\Theta\right)  \ .
\]
Denoting by $\mathbf{n}\left(  P\right)  $ the inward unit normal to $\partial
C$ at $P$ we have $\Theta=\mathbf{n}\left(  P\right)  $ and
\[
\gamma_{\Theta}\left(  \delta\right)  =\left\{  x\in C:x\cdot\mathbf{n}\left(
P\right)  =\inf_{y\in C}\left(  y\cdot\Theta\right)  +\delta\right\}  .
\]
The normal $\mathbf{n}\left(  P\right)  $ splits the chord into two parts
$\gamma_{\mathbf{n}\left(  P\right)  }^{+}\left(  \delta\right)  $ and
$\gamma_{\mathbf{n}\left(  P\right)  }^{-}\left(  \delta\right)  $ of lengths
$\left\vert \gamma_{\mathbf{n}\left(  P\right)  }^{+}\left(  \delta\right)
\right\vert $ and $\left\vert \gamma_{\mathbf{n}\left(  P\right)  }^{-}\left(
\delta\right)  \right\vert $ respectively (see Figure \ref{chord88}).
\end{definition}

\begin{figure}[h]
\begin{center}
\includegraphics[width=55mm]{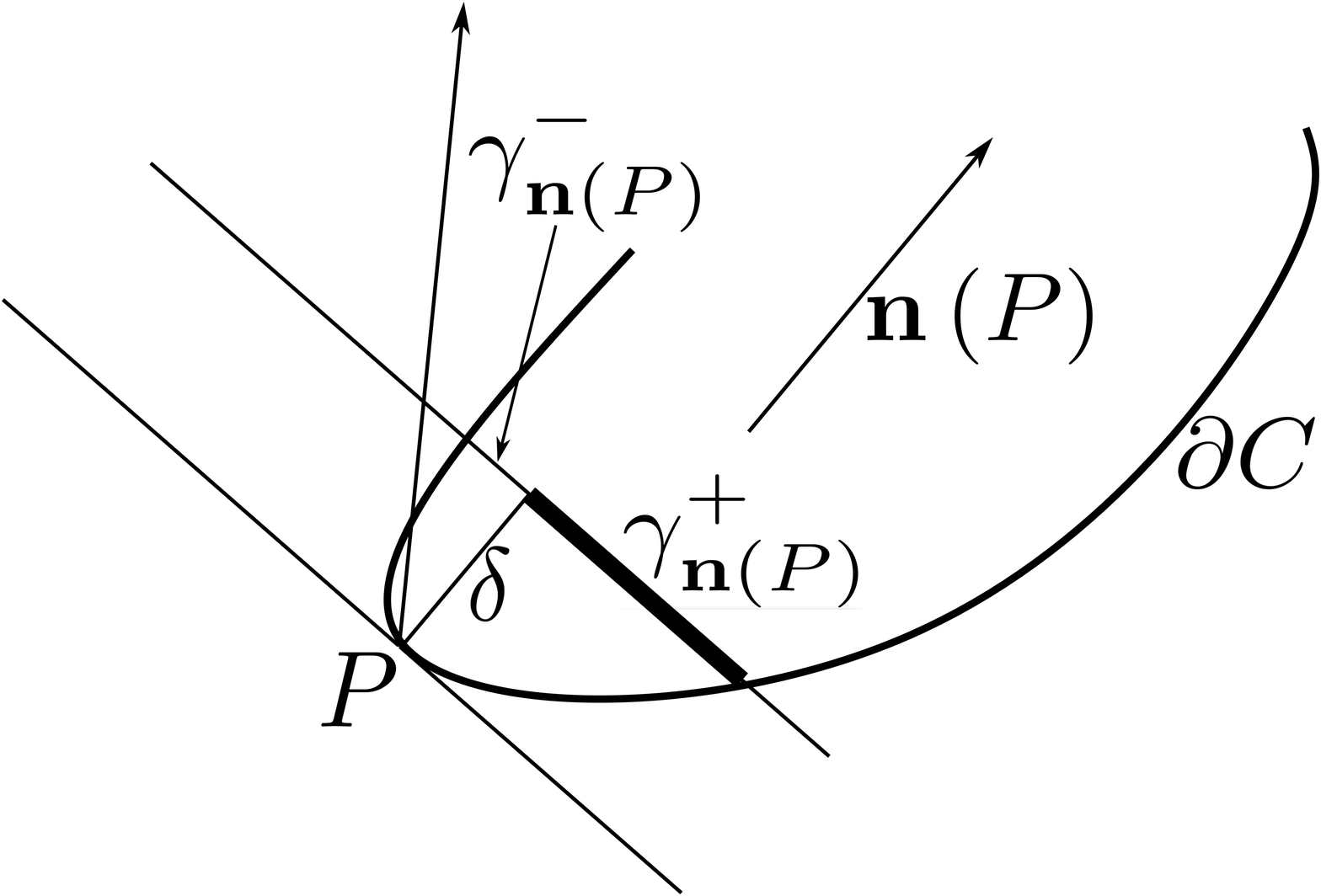}
\end{center}
\caption{Replacing $\left\vert \gamma_{\mathbf{n}\left(  P\right)  }\left(
\delta\right)  \right\vert $ with $\min\left\{  \left\vert \gamma
_{\mathbf{n}\left(  P\right)  }^{-}\left(  \delta\right)  \right\vert
,\left\vert \gamma_{\mathbf{n}\left(  P\right)  }^{+}\left(  \delta\right)
\right\vert \right\}  $.}%
\label{chord88}%
\end{figure}

\begin{remark}
\label{remarkdfdf} Let $C$ be a convex planar body where $\partial C$ is
$\mathcal{C}^{1}$. Then the existence of a positive constant $c$ such that,
for every $\theta$,
\[
\min\left\{  \left\vert \gamma_{\Theta}^{-}\left(  \delta\right)  \right\vert
,\left\vert \gamma_{\Theta}^{+}\left(  \delta\right)  \right\vert \right\}
\geqslant c\delta^{1/2}%
\]
is readily seen to be equivalent the existence of a positive number $R>0$ such
that for every $P\in\partial C$ there exists a disk of radius $R$ contained in
$C$ and tangent to $\partial C$ at $P$. This is a uniform version of the inner
disk condition that is used in the study of maximum principles for partial
differential equations (see e.g. \cite[Chapter 6]{Evans}).
\end{remark}

\begin{theorem}
\label{Equivalenza}Let $C\subset\mathbb{T}^{2}$ be a convex body and let
$0<\alpha\leqslant1$. Then the following are equivalent.

\begin{enumerate}
\item[a)] There exist constants $c>0$ and $\delta_{0}>0$ such that, for every
direction $\Theta$ and for $0<\delta<\delta_{0}$%
\begin{equation}
\min\left\{  \left\vert \gamma_{\Theta}^{-}\left(  \delta\right)  \right\vert
,\left\vert \gamma_{\Theta}^{+}\left(  \delta\right)  \right\vert \right\}
\geqslant c\delta^{1/\left(  1+\alpha\right)  }. \label{condizione corda}%
\end{equation}

\item[b)] The arc length parameterization $\Gamma\left(  s\right)  $ of
$\partial C$ is $C^{1,\alpha}$, that is $\ \Gamma\left(  s\right)
\in\mathcal{C}^{1}$ and there exists $M>0$ such that for every $s_{1},s_{2}$%
\[
\left\vert \Gamma^{\prime}\left(  s_{1}\right)  -\Gamma^{\prime}\left(
s_{2}\right)  \right\vert \leqslant M\left\vert s_{1}-s_{2}\right\vert
^{\alpha}.
\]

\end{enumerate}
\end{theorem}

\section{Geometric estimates of Fourier transforms\label{Sezione Geometric}}

Let $C$ be a planar convex body. The Fourier transform of its characteristic
function is defined by%
\[
\widehat{\chi}_{C}\left(  \xi\right)  =\int_{\mathbb{R}^{2}}\chi_{C}\left(
t\right)  e^{-2\pi i\xi\cdot t}dt.
\]
For a given direction $\Theta$ we are interested in the decay of
$\widehat{\chi}_{C}\left(  \rho\Theta\right)  $ as $\rho\rightarrow+\infty$.
Without loss of generality we can assume $\Theta=\left(  1,0\right)  $ so that%
\[
\widehat{\chi}_{C}\left(  \rho,0\right)  =\int_{\mathbb{-\infty}}^{+\infty
}\int_{\mathbb{-\infty}}^{+\infty}\chi_{C}\left(  t_{1},t_{2}\right)  e^{-2\pi
i\rho t_{1}}dt_{1}dt_{2}=\int_{\mathbb{-\infty}}^{+\infty}g\left(
t_{1}\right)  e^{-2\pi i\rho t_{1}}dt_{1}%
\]
where%
\[
g\left(  t_{1}\right)  =\int_{\mathbb{-\infty}}^{+\infty}\chi_{C}\left(
t_{1},t_{2}\right)  dt_{2}.
\]
Since $C$ is convex then $g\left(  t_{1}\right)  $ is supported and concave on
a suitable interval $\left[  A,B\right]  $. A change of variables allows us to
replace $g$ with a function $f$ which is positive, supported and concave on
the interval $\left[  -1,1\right]  $. The Fourier transform of $f$ is defined
as follows%
\[
\widehat{f}\left(  s\right)  =\int_{-\infty}^{+\infty}f\left(  x\right)
e^{-2\pi isx}dx.
\]

\begin{definition}
\label{Def delta}Let $f:\mathbb{R\rightarrow R}$ supported, nonnegative and
concave in the interval $\left[  -1,1\right]  $. For every $h\in\left(
-\frac{1}{2},\frac{1}{2}\right)  $ define%
\[
\mu_{f}\left(  h\right)  =\max\left\{  f\left(  -1+\left\vert h\right\vert
\right)  ,f\left(  1-\left\vert h\right\vert \right)  \right\}  .
\]

\end{definition}

The following upper bound of $\widehat{f}$ is due to A. Podkorytov \cite{P1}
(see also \cite{BRT}):%
\begin{equation}
\left\vert \widehat{f}\left(  s\right)  \right\vert \leqslant\left\vert
s\right\vert ^{-1}\,\mu_{f}\left(  \left\vert s\right\vert ^{-1}\right)  .
\label{Podkorytov}%
\end{equation}
This readily implies that%
\begin{equation}
\left\vert \widehat{\chi}_{C}\left(  \rho\Theta\right)  \right\vert
\leqslant\frac{c}{\rho}\left(  \left\vert \gamma_{\Theta}\left(  \rho
^{-1}\right)  \right\vert +\left\vert \gamma_{-\Theta}\left(  \rho
^{-1}\right)  \right\vert \right)  . \label{Stima chi hat corde}%
\end{equation}
In this section we estimate (an average of) $\widehat{f}\left(  s\right)  $
from below. Our approach is based on second order differences.

\subsection{Second order differences and moduli of smoothness}

Following \cite[Chapter 2]{DeV-L} we define the modulus of smoothness of a
function $\phi$ as follows.

\begin{definition}
Let $\phi\in L^{2}\left(  \mathbb{R}\right)  $ and, for every $h\in\mathbb{R}%
$, let $\Delta_{h}$ be the difference operator%
\[
\Delta_{h}\phi\left(  x\right)  =\phi\left(  x+h\right)  -\phi\left(
x\right)
\]
and let $\Delta_{h}^{2}$ be the second order difference operator
\[
\Delta_{h}^{2}\phi\left(  x\right)  =\Delta_{h}\Delta_{h}\phi\left(  x\right)
=\phi\left(  x+2h\right)  -2\phi\left(  x+h\right)  +\phi\left(  x\right)  .
\]

\end{definition}

\begin{definition}
Let $\phi\in L^{2}\left(  \mathbb{R}\right)  $ and let $\nu\geqslant0$. The
second $L^{2}$-modulus of smoothness of $\phi$ is given by%
\[
\omega_{2}\left(  \phi,\nu\right)  =\sup_{\left\vert h\right\vert \leqslant
\nu}\left\{  \int_{-\infty}^{+\infty}\left[  \Delta_{h}^{2}\phi\left(
x\right)  \right]  ^{2}dx\right\}  ^{1/2}.
\]

\end{definition}

Let $f$ be supported, nonnegative and concave in the interval $\left[
-1,1\right]  $. A relation between $\Delta_{h}^{2}f\left(  x\right)  $ and
$\mu_{f}\left(  h\right)  $ is proved in the following proposition.

\begin{proposition}
\label{StimaBilatera}Let $f:\mathbb{R\rightarrow R}$ supported, nonnegative
and concave in the interval $\left[  -1,1\right]  $. Then there exist
constants $c_{1},c_{2}>0$, independent of $f$, such that for every
$h\in\left(  -\frac{1}{2},\frac{1}{2}\right)  $,%
\begin{equation}
c_{1}\left\vert h\right\vert ^{1/2}\mu_{f}\left(  h\right)  \leqslant\left\{
\int_{-\infty}^{+\infty}\left\vert \Delta_{h}^{2}f\left(  x\right)
\right\vert ^{2}dx\right\}  ^{1/2}\leqslant c_{2}\left\vert h\right\vert
^{1/2}\mu_{f}\left(  h\right)  . \label{DoppiaStima}%
\end{equation}
Moreover, for $0<\nu<\frac{1}{2}$%
\begin{equation}
c_{1}\nu^{1/2}\mu_{f}\left(  \nu\right)  \leqslant\omega_{2}\left(
f,\nu\right)  \leqslant c_{2}\nu^{1/2}\mu_{f}\left(  \nu\right)  .
\label{DoppiaStimaOmega}%
\end{equation}

\end{proposition}

The proof needs the following lemma.

\begin{lemma}
Let $f$ be supported, nonnegative and concave in $\left[  -1,1\right]  $.
Then, if $0\leqslant\lambda_{1}<\lambda_{2}\leqslant1$ or $-1\leqslant
\lambda_{2}<\lambda_{1}\leqslant0$,%
\begin{equation}
f\left(  \lambda_{2}\right)  \leqslant2f\left(  \lambda_{1}\right)  ,
\label{lambda2lambda1}%
\end{equation}
In particular for every $x\in\mathbb{R}$%
\begin{equation}
f\left(  x\right)  \leqslant2f\left(  0\right)  \text{.} \label{2f(0)}%
\end{equation}
Moreover, for $0\leqslant\lambda_{1}<\lambda_{2}<1$%
\begin{equation}
f\left(  \lambda_{1}\right)  \leqslant\frac{1-\lambda_{1}}{1-\lambda_{2}%
}f\left(  \lambda_{2}\right)  \label{lambda1<lambda2-1}%
\end{equation}
and for $-1<\lambda_{2}<\lambda_{1}\leqslant0$%
\begin{equation}
f\left(  \lambda_{1}\right)  \leqslant\frac{1+\lambda_{1}}{1+\lambda_{2}%
}f\left(  \lambda_{2}\right)  \label{lambda1<lambda2-2}%
\end{equation}

\end{lemma}

\begin{proof}
We can clearly assume $0\leqslant\lambda_{1}<\lambda_{2}\leqslant1$. Since $f$
is concave in $\left[  -1,\lambda_{2}\right]  $ we have%
\[
f\left(  -1\right)  +\frac{f\left(  \lambda_{2}\right)  -f\left(  -1\right)
}{\lambda_{2}+1}\left(  \lambda_{1}+1\right)  \leqslant f\left(  \lambda
_{1}\right)  .
\]
This gives%
\[
f\left(  \lambda_{2}\right)  \left(  \lambda_{1}+1\right)  \leqslant\left(
\lambda_{2}+1\right)  f\left(  \lambda_{1}\right)
\]
and since $\frac{\lambda_{2}+1}{\lambda_{1}+1}\leqslant2$ we obtain $f\left(
\lambda_{2}\right)  \leqslant2f\left(  \lambda_{1}\right)  $. Similarly, since
$f$ is concave in $\left[  \lambda_{1},1\right]  $, we obtain%
\[
f\left(  1\right)  +\frac{f\left(  \lambda_{1}\right)  -f\left(  1\right)
}{\lambda_{1}-1}\left(  \lambda_{2}-1\right)  \leqslant f\left(  \lambda
_{2}\right)
\]
so that%
\[
f\left(  \lambda_{1}\right)  \frac{\lambda_{2}-1}{\lambda_{1}-1}\leqslant
f\left(  \lambda_{2}\right)  -f\left(  1\right)  \left(  \frac{\lambda
_{1}-\lambda_{2}}{\lambda_{1}-1}\right)  \leqslant f\left(  \lambda
_{2}\right)  .
\]
Then we obtain (\ref{lambda1<lambda2-1}).
\end{proof}

\begin{proof}
[Proof of Proposition \ref{StimaBilatera}]First of all observe that it is
enough to consider the case $h>0$. Indeed, the case $h<0$ follows applying
(\ref{DoppiaStima}) to the function $f\left(  -x\right)  $. Then we have%
\begin{align*}
\int_{-\infty}^{+\infty}\left\vert \Delta_{h}^{2}f\left(  x\right)
\right\vert ^{2}dx=  &  \int_{-1-2h}^{-1}\left\vert \Delta_{h}f\left(
x+h\right)  \right\vert ^{2}dx\\
&  +\int_{-1}^{1-2h}\left\vert \Delta_{h}f\left(  x+h\right)  -\Delta
_{h}f\left(  x\right)  \right\vert ^{2}dx+\int_{1-2h}^{1}\left\vert \Delta
_{h}f\left(  x\right)  \right\vert ^{2}dx\\
=  &  \mathcal{A}\left(  h\right)  +\mathcal{B}\left(  h\right)
+\mathcal{C}\left(  h\right)  \;.
\end{align*}
For the term $\mathcal{A}\left(  h\right)  $, using (\ref{lambda2lambda1}) we
obtain%
\begin{align*}
\mathcal{A}\left(  h\right)   &  =\int_{-1-2h}^{-1}\left\vert f\left(
x+2h\right)  -2f\left(  x+h\right)  \right\vert ^{2}dx\\
&  \leqslant2\int_{-1-2h}^{-1}\left[  f\left(  x+2h\right)  \right]
^{2}dx+8\int_{-1-2h}^{-1}\left[  f\left(  x+h\right)  \right]  ^{2}dx\\
&  =2\int_{-1}^{-1+2h}\left[  f\left(  x\right)  \right]  ^{2}dx+8\int
_{-1}^{-1+h}\left[  f\left(  x\right)  \right]  ^{2}dx.
\end{align*}

Observe that, by (\ref{lambda2lambda1}) and (\ref{lambda1<lambda2-2}), for
every if $x\in\left[  -1,-1+2h\right]  $ we have%
\[
f\left(  x\right)  \leqslant2f\left(  -1+2h\right)  \leqslant4f\left(
-1+h\right)  .
\]
Then
\[
\mathcal{A}\left(  h\right)  \leqslant c\,h\left[  f\left(  -1+h\right)
\right]  ^{2}.
\]
Similarly for $\mathcal{C}\left(  h\right)  $ we have%
\[
\mathcal{C}\left(  h\right)  \leqslant c\,h\left[  f\left(  1-h\right)
\right]  ^{2}.
\]
Now let us consider $\mathcal{B}\left(  h\right)  $. Since $f$ is concave in
the interval $\left[  -1,1\right]  $, for any given $h>0$ the function
$\Delta_{h}f\left(  x\right)  $ is decreasing in $\left[  -1,1-h\right]  $.
Let $\alpha\in\left[  -1,1-h\right]  $ satisfy $\Delta_{h}f\left(  x\right)
\geqslant0$ for $x\in\left[  -1,\alpha\right]  $ and $\Delta_{h}f\left(
x\right)  \leqslant0$ for $x\in\left[  \alpha,1-h\right]  $. Assume first that%
\begin{equation}
-1+h\leqslant\alpha\leqslant1-2h. \label{intervallo alpha}%
\end{equation}
Then%
\begin{align}
\mathcal{B}\left(  h\right)   &  \mathcal{=}\int_{-1}^{\alpha-h}\left\vert
\Delta_{h}f\left(  x+h\right)  -\Delta_{h}f\left(  x\right)  \right\vert
^{2}dx+\int_{\alpha-h}^{\alpha}\left\vert \Delta_{h}f\left(  x+h\right)
-\Delta_{h}f\left(  x\right)  \right\vert ^{2}dx\label{Split B}\\
&  +\int_{\alpha}^{1-2h}\left\vert \Delta_{h}f\left(  x+h\right)  -\Delta
_{h}f\left(  x\right)  \right\vert ^{2}dx\nonumber\\
&  =\mathcal{B}_{1}\left(  h\right)  \mathcal{+B}_{2}\left(  h\right)
\mathcal{+B}_{3}\left(  h\right)  .\nonumber
\end{align}
To estimate the term $\mathcal{B}_{1}\left(  h\right)  $ we use the inequality%
\[
\left\vert x-y\right\vert ^{2}\leqslant\left\vert x^{2}-y^{2}\right\vert
\]
that holds for $xy\geqslant0.$ Thus%
\begin{align*}
\mathcal{B}_{1}\left(  h\right)   &  \leqslant\int_{-1}^{\alpha-h}\left[
\left[  \Delta_{h}f\left(  x\right)  \right]  ^{2}-\left[  \Delta_{h}f\left(
x+h\right)  \right]  ^{2}\right]  dx\\
&  =\int_{-1}^{\alpha-h}\left[  \Delta_{h}f\left(  x\right)  \right]
^{2}dx-\int_{-1+h}^{\alpha}\left[  \Delta_{h}f\left(  x\right)  \right]
^{2}dx\\
&  =\int_{-1}^{-1+h}\left[  \Delta_{h}f\left(  x\right)  \right]  ^{2}%
dx-\int_{\alpha-h}^{\alpha}\left[  \Delta_{h}f\left(  x\right)  \right]
^{2}dx\\
&  \leqslant\int_{-1}^{-1+h}\left[  \Delta_{h}f\left(  x\right)  \right]
^{2}dx=\int_{-1}^{-1+h}\left[  f\left(  x+h\right)  -f\left(  x\right)
\right]  ^{2}dx.
\end{align*}
Using (\ref{lambda2lambda1}) and (\ref{lambda1<lambda2-2}) the latter can be
bounded by%
\begin{align*}
&  2\int_{-1}^{-1+h}\left[  f\left(  x+h\right)  \right]  ^{2}dx+2\int
_{-1}^{-1+h}\left[  f\left(  x\right)  \right]  ^{2}dx\\
&  \leqslant8\int_{-1}^{-1+h}\left[  f\left(  -1+2h\right)  \right]
^{2}dx+8\int_{-1}^{-1+h}\left[  f\left(  -1+h\right)  \right]  ^{2}dx\\
&  \leqslant8h\left(  \left[  f\left(  -1+2h\right)  \right]  ^{2}+\left[
f\left(  -1+h\right)  \right]  ^{2}\right)  \leqslant40h\,\left[  f\left(
-1+h\right)  \right]  ^{2}.
\end{align*}
A similar estimate holds for $\mathcal{B}_{3}\left(  h\right)  $. To estimate
$\mathcal{B}_{2}\left(  h\right)  $ observe $\left[  \Delta_{h}f\left(
x\right)  \right]  ^{2}$ is decreasing for $-1\leqslant x\leqslant\alpha$ and
increasing for $\alpha\leqslant x\leqslant1-h$. Then, recalling
(\ref{intervallo alpha}), we have%
\begin{align*}
\mathcal{B}_{2}\left(  h\right)   &  \leqslant2\int_{\alpha-h}^{\alpha}\left[
\Delta_{h}f\left(  x+h\right)  \right]  ^{2}dx+2\int_{\alpha-h}^{\alpha
}\left[  \Delta_{h}f\left(  x\right)  \right]  ^{2}dx\\
&  =2\int_{\alpha}^{\alpha+h}\left[  \Delta_{h}f\left(  x\right)  \right]
^{2}dx+2\int_{\alpha-h}^{\alpha}\left[  \Delta_{h}f\left(  x\right)  \right]
^{2}dx\\
&  \leqslant2\int_{1-h}^{1}[\Delta_{h}f(t+\alpha+h-1)]^{2}dt+2\int_{-1}%
^{-1+h}[\Delta_{h}f(t+\alpha-h+1)]^{2}dt.
\end{align*}
Observe that if $t\in\left[  1-h,1\right]  $ we have
\[
\alpha\leqslant t+\alpha+h-1<t,
\]
so that%
\[
\Delta_{h}f(t+\alpha+h-1)]^{2}\leqslant\left[  \Delta_{h}f\left(  t\right)
\right]  ^{2}.
\]
This gives%
\[
\int_{1-h}^{1}[\Delta_{h}f(t+\alpha+h-1)]^{2}dt\leqslant\int_{1-h}^{1}\left[
\Delta_{h}f\left(  t\right)  \right]  ^{2}dt\leqslant c\,h\left[  f\left(
1-h\right)  \right]  ^{2}.
\]
Similarly for $t\in\left[  -1,-1+h\right]  $ we have%
\[
t\leqslant t+\alpha-h+1\leqslant\alpha,
\]
so that%
\[
\int_{-1}^{-1+h}[\Delta_{h}f(t+\alpha-h+1)]^{2}dt\leqslant\int_{-1}%
^{-1+h}[\Delta_{h}f(t)]^{2}dt\leqslant c\,h\left[  f\left(  1+h\right)
\right]  ^{2}.
\]
Therefore%
\[
\mathcal{B}_{2}\left(  h\right)  \leqslant c\,h\left[  f\left(  1-h\right)
\right]  ^{2}+c\,h\left[  f\left(  1+h\right)  \right]  ^{2}.
\]

Finally observe that when $-1\leqslant\alpha<-1+h$ or $1-2h<\alpha
\leqslant1-h$ equation (\ref{Split B}) reduces to two terms that can be
handled as in the previous case. The second inequality in (\ref{DoppiaStima})
is a consequence of the previous computations.

To prove the first inequality in (\ref{DoppiaStima}) observe that%
\begin{align*}
\int_{-\infty}^{+\infty}\left[  \Delta_{h}^{2}f\left(  x\right)  \right]
^{2}dx  &  \geqslant\int_{1-h}^{1}\left[  \Delta_{h}^{2}f\left(  x\right)
\right]  ^{2}dx+\int_{-1-2h}^{-1-h}\left[  \Delta_{h}^{2}f\left(  x\right)
\right]  ^{2}dx\\
&  =\int_{1-h}^{1}\left[  f\left(  x\right)  \right]  ^{2}dx+\int
_{-1-2h}^{-1-h}\left[  f\left(  x+2h\right)  \right]  ^{2}dx\\
&  \geqslant c\,h\mu_{f}\left(  h\right)  ^{2}.
\end{align*}

This completes the proof of (\ref{DoppiaStima}) and proves also the first
inequality in (\ref{DoppiaStimaOmega}). To prove the second inequality in
(\ref{DoppiaStimaOmega}) let us fix $\nu$ and let $0\leqslant h\leqslant\nu$.
Then by (\ref{lambda2lambda1}) we have%
\[
\mu_{f}\left(  h\right)  \leqslant2\mu_{f}\left(  \nu\right)
\]
so that%
\[
\left\{  \int_{-\infty}^{+\infty}\left\vert \Delta_{h}^{2}f\left(  x\right)
\right\vert ^{2}dx\right\}  ^{1/2}\leqslant c_{2}\left\vert h\right\vert
^{1/2}\mu_{f}\left(  h\right)  \leqslant2c_{2}\nu^{1/2}\mu_{f}\left(
\nu\right)  .
\]
Therefore%
\[
\omega_{2}\left(  f,\nu\right)  \leqslant c\,\delta^{1/2}\mu_{f}\left(
\nu\right)  .
\]

\end{proof}

\subsection{Decay of Fourier transforms}

Moduli of smoothness turn out to be a link between Fourier transforms and the
chord estimates as introduced in Definition \ref{firstdef}. The following
result is known (see \cite{BCT}, \cite{Gioev}, \cite{Br-Pi}). We give a proof
for completeness.

\begin{lemma}
\label{Stime trasformata}There exists $c>0$ such that for every $\phi\in
L^{2}\left(  \mathbb{R}\right)  $ and $\rho\geqslant1$ we have%
\[
\left\{  \int_{\left\{  \rho\leqslant\left\vert s\right\vert \right\}
}\left\vert \widehat{\phi}\left(  s\right)  \right\vert ^{2}ds\right\}
^{1/2}\leqslant c\omega_{2}\left(  \phi,\rho^{-1}\right)
\]
and%
\begin{equation}
\left\{  \int_{\left\{  \left\vert s\right\vert \leqslant\rho\right\}
}\left\vert s\right\vert ^{4}\left\vert \widehat{\phi}\left(  s\right)
\right\vert ^{2}ds\right\}  ^{1/2}\leqslant c\rho^{2}\omega_{2}\left(
\phi,\rho^{-1}\right)  . \label{xi piccolo}%
\end{equation}

\end{lemma}

\begin{proof}
Let $\eta\in\mathcal{S}\left(  \mathbb{R}\right)  $ satisfy $\widehat{\eta
}\left(  0\right)  =1$ and $\widehat{\eta}\left(  s\right)  =0$ for
$\left\vert s\right\vert \geqslant1$. Let%
\[
V_{\rho}\left(  x\right)  =2\rho\eta\left(  \rho x\right)  -\frac{\rho}{2}%
\eta\left(  \frac{\rho x}{2}\right)  .
\]
Since%
\[
\widehat{V}_{\rho}\left(  s\right)  =2\widehat{\eta}\left(  \rho^{-1}s\right)
-\widehat{\eta}\left(  2\rho^{-1}s\right)
\]
it follows that $\widehat{V_{\rho}}\left(  s\right)  =0$ if $\left\vert
s\right\vert \geqslant\rho$. Then Plancherel Theorem gives%
\begin{align*}
\int_{\left\{  \left\vert s\right\vert \geqslant\rho\right\}  }\left\vert
\widehat{\phi}\left(  s\right)  \right\vert ^{2}ds  &  \leqslant
\int_{\mathbb{R}}\left\vert \left(  1-\widehat{V_{\rho}}\left(  s\right)
\right)  \widehat{\phi}\left(  s\right)  \right\vert ^{2}ds\\
&  =\int_{\mathbb{R}}\left\vert \phi\left(  x\right)  -V_{\rho}\ast\phi\left(
x\right)  \right\vert ^{2}dx.
\end{align*}
Since%
\begin{align*}
V_{\rho}\ast\phi\left(  x\right)   &  =\int_{\mathbb{R}}\phi\left(
x-y\right)  \left(  2\rho\eta\left(  \rho y\right)  -\frac{\rho}{2}\eta\left(
\frac{\rho y}{2}\right)  \right)  dy\\
&  =\int_{\mathbb{R}}2\phi\left(  x-y\right)  \rho\eta\left(  \rho y\right)
dy-\int_{\mathbb{R}}\phi\left(  x-y\right)  \frac{\rho}{2}\eta\left(
\frac{\rho y}{2}\right)  dy\\
&  =\int_{\mathbb{R}}\left[  2\phi\left(  x-\rho^{-1}z\right)  -\phi\left(
x-2\rho^{-1}z\right)  \right]  \eta\left(  z\right)  dz
\end{align*}
and $\int_{\mathbb{R}}\eta\left(  z\right)  dz=\widehat{\eta}\left(  0\right)
=1$, we have%
\begin{align*}
\phi\left(  x\right)  -V_{\rho}\ast\phi\left(  x\right)   &  =\int
_{\mathbb{R}}\phi\left(  x\right)  \eta\left(  z\right)  dz-\int_{\mathbb{R}%
}\left[  2\phi\left(  x-\rho^{-1}z\right)  -\phi\left(  x-2\rho^{-1}z\right)
\right]  \eta\left(  z\right)  dz\\
&  =\int_{\mathbb{R}}\left[  \phi\left(  x\right)  -2\phi\left(  x-\rho
^{-1}z\right)  +\phi\left(  x-2\rho^{-1}z\right)  \right]  \eta\left(
z\right)  dz\\
&  =\int_{\mathbb{R}}\Delta_{-\rho^{-1}z}^{2}\phi\left(  x\right)  \eta\left(
z\right)  dz.
\end{align*}
Since%
\begin{equation}
\omega_{2}\left(  \phi,\rho^{-1}\left\vert z\right\vert \right)
\leqslant\left(  1+\left\vert z\right\vert ^{2}\right)  \,\omega_{2}\left(
\phi,\rho^{-1}\right)  , \label{DeVore}%
\end{equation}
see \cite[Chapter 2, \S 7]{DeV-L}, then Minkowski integral inequality yields%
\begin{align*}
\left\{  \int_{\left\{  \left\vert s\right\vert \geqslant\rho\right\}
}\left\vert \widehat{\phi}\left(  s\right)  \right\vert ^{2}ds\right\}
^{1/2}  &  =\left\{  \int_{\mathbb{R}}\left\vert \int_{\mathbb{R}}%
\Delta_{-\rho^{-1}z}^{2}\phi\left(  x\right)  \eta\left(  z\right)
dz\right\vert ^{2}dx\right\}  ^{1/2}\\
&  \leqslant\int_{\mathbb{R}}\left\{  \int_{\mathbb{R}}\left\vert
\Delta_{-\rho^{-1}z}^{2}\phi\left(  x\right)  \right\vert ^{2}dx\right\}
^{1/2}\left\vert \eta\left(  z\right)  \right\vert dz\\
&  \leqslant\int_{\mathbb{R}}\omega_{2}\left(  \phi,\rho^{-1}\left\vert
z\right\vert \right)  \left\vert \eta\left(  z\right)  \right\vert dz\\
&  \leqslant\omega_{2}\left(  \phi,\rho^{-1}\right)  \int_{\mathbb{R}}\left(
1+\left\vert z\right\vert \right)  ^{2}\left\vert \eta\left(  z\right)
\right\vert dz\\
&  \leqslant c\,\omega_{2}\left(  \phi,\rho^{-1}\right)  .
\end{align*}
To prove (\ref{xi piccolo}) let $h=\left(  4\pi\rho\right)  ^{-1}$. Then, for
$\left\vert s\right\vert \leqslant\rho$, we have $\left\vert 2\pi
sh\right\vert \leqslant c\left\vert e^{2\pi ish}-1\right\vert $, so that%
\begin{align*}
\int_{\left\{  \left\vert s\right\vert \leqslant\rho\right\}  }\left\vert
s\right\vert ^{4}\left\vert \widehat{\phi}\left(  s\right)  \right\vert
^{2}ds  &  =2^{4}\rho^{4}\int_{\left\{  \left\vert s\right\vert \leqslant
\rho\right\}  }\left\vert 2\pi sh\right\vert ^{4}\left\vert \widehat{\phi
}\left(  s\right)  \right\vert ^{2}ds\\
&  \leqslant c\rho^{4}\int_{\left\{  \left\vert s\right\vert \leqslant
\rho\right\}  }\left\vert e^{2\pi ish}-1\right\vert ^{4}\left\vert
\widehat{\phi}\left(  s\right)  \right\vert ^{2}ds\\
&  =c\rho^{4}\int_{\mathbb{R}}\left\vert \left(  e^{2\pi ish}-1\right)
^{2}\,\widehat{\phi}\left(  s\right)  \right\vert ^{2}ds\\
&  =c\rho^{4}\int_{\mathbb{R}}\left\vert \widehat{\Delta_{h}^{2}\phi}\left(
s\right)  \right\vert ^{2}ds\leqslant c\rho^{4}\omega_{2}\left(  \phi
,\rho^{-1}\right)  ^{2}.
\end{align*}

\end{proof}

\begin{lemma}
\label{Lemma Media}There exist four positive constants $\alpha,\beta
,c_{1},c_{2}$, such that, for every $f:\mathbb{R\rightarrow R}$ supported,
nonnegative and concave in the interval $\left[  -1,1\right]  $ and every
$\rho\geqslant2\alpha^{-1}$, we have%
\[
\frac{c_{1}}{\rho}\mu_{f}\left(  \rho^{-1}\right)  ^{2}\leqslant\int
_{\alpha\rho\leqslant\left\vert s\right\vert \leqslant\beta\rho}\left\vert
\widehat{f}\left(  s\right)  \right\vert ^{2}ds\leqslant\frac{c_{2}}{\rho}%
\mu_{f}\left(  \rho^{-1}\right)  ^{2}%
\]
where $\mu_{f}\left(  \rho^{-1}\right)  $ comes from Definition
\ref{Def delta}.
\end{lemma}

The upper bound follows from (\ref{Podkorytov}). The lower bound has been
first proved by Podkorytov \cite{P2}. We provide an alternative proof that
depends on the previous lemma and may be of independent interest.

\begin{proof}
Let $\alpha<\beta$, by Lemma \ref{Stime trasformata}, Proposition
\ref{StimaBilatera} and (\ref{DeVore}) we have%
\begin{align*}
\int_{\alpha\rho\leqslant\left\vert s\right\vert \leqslant\beta\rho}\left\vert
\widehat{f}\left(  s\right)  \right\vert ^{2}ds  &  \leqslant\int_{\alpha
\rho\leqslant\left\vert s\right\vert }\left\vert \widehat{f}\left(  s\right)
\right\vert ^{2}ds\leqslant c\left[  \omega_{2}\left(  \phi,\alpha^{-1}%
\rho^{-1}\right)  \right]  ^{2}\\
&  \leqslant c\left(  1+\alpha^{-1}\right)  ^{2}\left[  \omega_{2}\left(
\phi,\rho^{-1}\right)  \right]  ^{2}\leqslant c\rho^{-1}\mu_{f}\left(
\rho^{-1}\right)  ^{2}.
\end{align*}
Let $\left\vert h\right\vert \leqslant\rho^{-1}$. Using Lemma
\ref{Stime trasformata} we obtain%
\begin{align*}
&  \int_{\mathbb{R}}\left\vert \Delta_{h}^{2}f\left(  x\right)  \right\vert
^{2}dx=\int_{\mathbb{R}}\left\vert e^{2\pi ihs}-1\right\vert ^{4}\left\vert
\widehat{f}\left(  s\right)  \right\vert ^{2}ds\\
\leqslant &  c\left\vert h\right\vert ^{4}\int_{\left\vert s\right\vert
\leqslant\alpha\rho}\left\vert s\right\vert ^{4}\left\vert \widehat{f}\left(
s\right)  \right\vert ^{2}ds+c\int_{\alpha\rho\leqslant\left\vert s\right\vert
\leqslant\beta\rho}\left\vert \widehat{f}\left(  s\right)  \right\vert
^{2}ds+c\int_{\left\vert s\right\vert \geqslant\beta\rho}\left\vert
\widehat{f}\left(  s\right)  \right\vert ^{2}ds\\
\leqslant &  c\alpha^{4}\omega_{2}\left(  f,\alpha^{-1}\rho^{-1}\right)
^{2}+c\int_{\alpha\rho\leqslant\left\vert s\right\vert \leqslant\beta\rho
}\left\vert \widehat{f}\left(  s\right)  \right\vert ^{2}ds+c\left[
\omega_{2}\left(  f,\beta^{-1}\rho^{-1}\right)  \right]  ^{2}.
\end{align*}
Hence, being $\left\vert h\right\vert \leqslant\rho^{-1}$,%
\begin{align*}
\omega_{2}\left(  f,\rho^{-1}\right)  ^{2}  &  \leqslant c\alpha^{4}\omega
_{2}\left(  f,\alpha^{-1}\rho^{-1}\right)  ^{2}\\
&  +c\int_{\alpha\rho\leqslant\left\vert s\right\vert \leqslant\beta\rho
}\left\vert \widehat{f}\left(  s\right)  \right\vert ^{2}ds+c\left[
\omega_{2}\left(  f,\beta^{-1}\rho^{-1}\right)  \right]  ^{2}.
\end{align*}
Since $\alpha^{-1}\rho^{-1}\leqslant1/2$ and $\beta^{-1}\rho^{-1}\leqslant
1/2$, by Proposition \ref{StimaBilatera} we have%
\begin{align*}
\rho^{-1}\mu_{f}\left(  \rho^{-1}\right)  ^{2}  &  \leqslant c\alpha^{3}%
\rho^{-1}\mu_{f}\left(  \alpha^{-1}\rho^{-1}\right)  ^{2}\\
&  +c\int_{\alpha\rho\leqslant\left\vert s\right\vert \leqslant\beta\rho
}\left\vert \widehat{f}\left(  s\right)  \right\vert ^{2}ds+c\beta^{-1}%
\rho^{-1}\mu_{f}\left(  \beta^{-1}\rho^{-1}\right)  ^{2}\\
&  \leqslant c\alpha\rho^{-1}\left[  \alpha\mu_{f}\left(  \alpha^{-1}\rho
^{-1}\right)  \right]  ^{2}\\
&  +c\int_{\alpha\rho\leqslant\left\vert s\right\vert \leqslant\beta\rho
}\left\vert \widehat{f}\left(  s\right)  \right\vert ^{2}ds+c\beta^{-1}%
\rho^{-1}\mu_{f}\left(  \beta^{-1}\rho^{-1}\right)  ^{2}\\
&  \leqslant c\alpha\rho^{-1}\mu_{f}\left(  \rho^{-1}\right)  ^{2}%
+c\int_{\alpha\rho\leqslant\left\vert s\right\vert \leqslant\beta\rho
}\left\vert \widehat{f}\left(  s\right)  \right\vert ^{2}ds+c\beta^{-1}%
\rho^{-1}\mu_{f}\left(  \rho^{-1}\right)  ^{2},
\end{align*}
because (\ref{lambda1<lambda2-1}) yields%
\[
\mu_{f}\left(  \alpha^{-1}\rho^{-1}\right)  \leqslant\alpha^{-1}\mu_{f}\left(
\rho^{-1}\right)
\]
and (\ref{lambda2lambda1}) gives%
\[
\mu_{f}\left(  \beta^{-1}\rho^{-1}\right)  \leqslant2\mu_{f}\left(  \rho
^{-1}\right)  .
\]
It follows that%
\[
\left(  1-c\alpha-c\beta^{-1}\right)  \rho^{-1}\mu_{f}\left(  \rho
^{-1}\right)  ^{2}\leqslant c\int_{\alpha\rho\leqslant\left\vert s\right\vert
\leqslant\beta\rho}\left\vert \widehat{f}\left(  s\right)  \right\vert
^{2}ds.
\]
Letting $\alpha$ sufficiently small and $\beta$ sufficiently large gives%
\[
\rho^{-1}\mu_{f}\left(  \rho^{-1}\right)  ^{2}\leqslant c\int_{\alpha
\rho\leqslant\left\vert s\right\vert \leqslant\beta\rho}\left\vert \widehat
{f}\left(  s\right)  \right\vert ^{2}ds
\]
for every $\rho\geqslant2\alpha^{-1}$.
\end{proof}

\begin{remark}
\label{Da sotto solo positivi}Since in the above lemma $f$ is real we have
$\left\vert \widehat{f}\left(  -s\right)  \right\vert =\left\vert \widehat
{f}\left(  s\right)  \right\vert $ and therefore we also obtain%
\[
\rho^{-1}\mu_{f}\left(  \rho^{-1}\right)  ^{2}\leqslant c\int_{\alpha\rho
}^{\beta\rho}\left\vert \widehat{f}\left(  s\right)  \right\vert ^{2}ds.
\]

\end{remark}

\begin{theorem}
\label{StimaDaSotto}Let $C\subset\mathbb{T}^{2}$ be a convex body, let
$\sigma\in\left[  \frac{1}{2},1\right]  $ and let $\delta_{0}>0$. Let $I$ be
an interval in $\mathbb{T}$ and let $\Theta=\left(  \cos\theta,\sin
\theta\right)  $ with $\theta\in I$. Assume there exists a constant $c_{1}>0$
such that every $0<\delta\leqslant\delta_{0}$ we have%
\[
\left\vert \gamma_{-\Theta}\left(  \delta\right)  \right\vert +\left\vert
\gamma_{\Theta}\left(  \delta\right)  \right\vert \geqslant c_{1}%
\delta^{\sigma}.
\]
Then there exist positive constants $c_{2},c_{3}$, independent of $\theta\in
I$, such that for every $\rho\geqslant c_{3}$,%
\[
\left\{  \int_{1/2\leqslant\left\vert \tau\right\vert \leqslant1}\left\vert
\widehat{\chi}_{C}\left(  \tau\rho\Theta\right)  \right\vert ^{2}%
d\tau\right\}  ^{1/2}\geqslant c_{2}\rho^{-1-\sigma}.
\]

\end{theorem}

\begin{proof}
Assume, without loss of generality, that $\Theta=\left(  1,0\right)  $. Then%
\[
\widehat{\chi}_{C}\left(  \rho,0\right)  =\widehat{g}\left(  \rho\right)
\]
with%
\[
g\left(  t_{1}\right)  =\int_{\mathbb{R}}\chi_{C}\left(  t_{1},t_{2}\right)
dt_{2}.
\]
Observe that $g$ is nonnegative, supported and concave in a suitable interval
$\left[  A,B\right]  $. Let%
\[
f\left(  x\right)  =g\left(  \frac{A+B}{2}+x\frac{B-A}{2}\right)
\]
then%
\[
\widehat{f}\left(  s\right)  =\frac{2}{B-A}e^{2\pi is\frac{A+B}{B-A}%
}\,\widehat{g}\left(  \frac{2s}{B-A}\right)  .
\]
By Lemma \ref{Lemma Media} and Remark \ref{Da sotto solo positivi} we have%
\[
\frac{c_{1}}{\rho^{2}}\mu_{f}\left(  \rho^{-1}\right)  ^{2}\leqslant\frac
{1}{\rho}\int_{\alpha\rho}^{\beta\rho}\left\vert \widehat{f}\left(  s\right)
\right\vert ^{2}ds,
\]
so that%
\[
\frac{c}{\rho}\left(  f\left(  -1+\rho^{-1}\right)  +f\left(  1-\rho
^{-1}\right)  \right)  \leqslant\left\{  \int_{\alpha}^{\beta}\left\vert
\widehat{f}\left(  \rho\tau\right)  \right\vert ^{2}d\tau\right\}  ^{1/2}.
\]
Hence%
\begin{align*}
\frac{1}{\rho}\left(  g\left(  A+\frac{B-A}{2\rho}\right)  +g\left(
B-\frac{B-A}{2\rho}\right)  \right)   &  \leqslant\frac{c}{B-A}\left\{
\int_{\alpha}^{\beta}\left\vert \widehat{g}\left(  \frac{2\rho\tau}%
{B-A}\right)  \right\vert ^{2}d\tau\right\}  ^{1/2}\\
&  =\frac{c}{\sqrt{B-A}}\left\{  \int_{\frac{2\alpha}{B-A}}^{\frac{2\beta
}{B-A}}\left\vert \widehat{g}\left(  \rho\omega\right)  \right\vert
^{2}d\omega\right\}  ^{1/2}.
\end{align*}
Since $C$ is a convex body there exists a positive constant $\sigma$ (only
depending on $C$) such that%
\[
\sigma\leqslant B-A\leqslant\operatorname{diam}\left(  C\right)  .
\]
Then for suitable constants $c,\alpha^{\prime}$ and $\beta^{\prime}$ we have%
\[
\frac{1}{\rho}\left(  g\left(  A+\frac{B-A}{2\rho}\right)  +g\left(
B-\frac{B-A}{2\rho}\right)  \right)  \leqslant c\left\{  \int_{\alpha^{\prime
}}^{\beta^{\prime}}\left\vert \widehat{g}\left(  \rho\omega\right)
\right\vert ^{2}d\omega\right\}  ^{1/2}.
\]
Since%
\[
g\left(  B-\frac{B-A}{2\rho}\right)  =\gamma_{\Theta}\left(  \frac{B-A}{2\rho
}\right)
\]
and%
\[
g\left(  A+\frac{B-A}{2\rho}\right)  =\gamma_{-\Theta}\left(  \frac{B-A}%
{2\rho}\right)
\]
with $\Theta=\left(  1,0\right)  $, we have%
\[
\frac{c}{\rho}\left(  \frac{B-A}{2\rho}\right)  ^{\sigma}\leqslant\left\{
\int_{\alpha^{\prime}}^{\beta^{\prime}}\left\vert \widehat{g}\left(
\rho\omega\right)  \right\vert ^{2}d\omega\right\}  ^{1/2}.
\]
Hence%
\[
\frac{c}{\rho^{1+\sigma}}\leqslant\left\{  \int_{\alpha^{\prime}}%
^{\beta^{\prime}}\left\vert \widehat{g}\left(  \rho\omega\right)  \right\vert
^{2}d\omega\right\}  ^{1/2}.
\]
A suitable change of variables gives%
\[
\frac{c}{\rho^{1+\sigma}}\leqslant\left\{  \int_{1/2}^{1}\,\left\vert
\widehat{g}\left(  \rho\omega\right)  \right\vert ^{2}d\omega\right\}  ^{1/2}%
\]
and then%
\[
\left\{  \int_{1/2}^{1}\left\vert \widehat{\chi}_{C}\left(  \rho
\omega,0\right)  \right\vert ^{2}d\omega\right\}  ^{1/2}\geqslant
c\rho^{-1-\sigma}.
\]

\end{proof}

\section{Proofs of the main results}

\begin{proof}
[Proof of Proposition \ref{stimadeltadasotto}]Let $P_{0}$ be a fixed interior
point of $C$, let%
\[
d=\inf_{P\in\partial C}\operatorname*{dist}\left(  P,P_{0}\right)
\]
and let $D$ be the disk of radious $d/2$ centered at $P_{0}$. Clearly
$D\subset C$. Let us fix a direction $\Theta$ and let $P\in\partial C$ be such
that%
\begin{equation}
P\cdot\Theta=\inf_{x\in C}x\cdot\Theta. \label{1}%
\end{equation}
Then%
\[
\gamma_{\Theta}\left(  \delta\right)  =\left\vert \left\{  x\in C:x\cdot
\Theta=P\cdot\Theta+\delta\right\}  \right\vert
\]
Without loss of generality we can assume that $P$ is the origin and that
$P_{0}=\left(  0,y_{0}\right)  $ for some $y_{0}\geqslant d$. Let
$P_{1}=\left(  -d/2,y_{0}\right)  $, $P_{2}=\left(  d/2,y_{0}\right)  $ and
let $T$ be the triangle with vertices $P,P_{1},P_{2}$ (see Figure
\ref{Figure3}). Since $T\subset C$ we have%
\[
\left\vert \gamma_{\Theta}\left(  \delta\right)  \right\vert \geqslant
\left\vert \left\{  x\in T:x\cdot\Theta=\delta\right\}  \right\vert .
\]
From (\ref{1}) we obtain%
\begin{equation}
P_{1}\cdot\Theta\geqslant0\hspace{1cm}\text{and}\hspace{1cm}P_{2}\cdot
\Theta\geqslant0. \label{2}%
\end{equation}
Let us write $P_{1}=\left\vert P_{1}\right\vert \left(  -\sin\gamma,\cos
\gamma\right)  $, $P_{2}=\left\vert P_{2}\right\vert \left(  \sin\gamma
,\cos\gamma\right)  $ and let $\Theta=\left(  \cos\theta,\sin\theta\right)  $.
Using (\ref{2}) we have $\gamma\leqslant\theta\leqslant\pi-\gamma$. Observe
that
\[
\tan\gamma=\frac{d}{2y_{0}}\geqslant\frac{d}{2\sup\limits_{P\in\partial
C}d\left(  P,P_{0}\right)  }%
\]
so that $\gamma\geqslant\gamma_{0}$ with $\gamma_{0}$ independent of $\Theta$
and $\delta$. Since $T$ is symmetric about the vertical axis it suffices to
consider the case $\gamma\leqslant\theta\leqslant\frac{\pi}{2}$. Let
$0<\delta\leqslant\frac{d}{2}\sin\left(  \gamma_{0}\right)  $ (this ensures
that the point $Q_{0}$ is inside $T$). Then%
\begin{align*}
&  \left\vert \left\{  x\in T:x\cdot\Theta=\delta\right\}  \right\vert
=\left\vert Q_{1}-Q_{2}\right\vert \geqslant\left\vert Q_{2}-Q_{0}\right\vert
\\
&  =\delta\left[  \tan\left(  \frac{\pi}{2}-\theta\right)  -\tan\left(
\frac{\pi}{2}-\theta-\gamma\right)  \right]  \geqslant\delta\tan
\gamma\geqslant\delta\tan\gamma_{0}.
\end{align*}
\begin{figure}[ptb]
\begin{tikzpicture}%
[line cap=round,line join=round,>=triangle 45,x=0.8cm,y=0.8cm]
\begin{axis}[x=0.8cm,y=0.8cm,axis lines=middle,xmin=-6,xmax=6,ymin=-1,ymax=8.5,ticks=none]
\fill[line width=0.1pt,color=zzttqq,fill=zzttqq,fill opacity=0.1] (0.,0.) -- (-2.,6.12) -- (2.,6.12) -- cycle;
\draw[line width=0.pt,color=black] (0,0) -- (0.:1) arc (0.:37.56:1) -- cycle;
\draw[line width=0.pt,color=black] (0,0) -- (71.9:1) arc (71.9:90.:1) -- cycle;
\draw[line width=0.5pt,color=zzttqq] (0.,0.)-- (-2.,6.12) -- (2.,6.12) -- cycle;
\draw [line width=0.5pt,dash pattern=on 2pt off 2pt,domain=-6:1.56] plot(\x,{(--1.96-1.11*\x)/0.854});
\draw[line width=0.5pt] (0.,0.)-- (1.11,0.854);
\draw[line width=1.pt] (-1.3,4)-- (0.53,1.612);
\draw[line width=0.5pt] (1.11,0.853)-- (0.53,1.612);
\begin{scriptsize}
\draw[fill=black] (0.,0.) circle (1pt);
\draw[color=black] (-0.2,-0.2) node {P};
\draw[fill=black] (0,6.12) circle (1pt);
\draw[color=black] (-0.2,6.4) node {$P_0$};
\draw[fill=black] (-2.,6.12) circle (1pt);
\draw[color=black] (-2.2,6.4) node {$P_1$};
\draw[fill=black] (2.,6.12) circle (1pt);
\draw[color=black] (2.2,6.4) node {$P_2$};
\draw[fill=black] (-1.3,4) circle (1pt);
\draw[color=black] (-1,4.1) node {$Q_1$};
\draw[fill=black] (0.,2.3) circle (1pt);
\draw[color=black] (0.25,2.6) node {$Q_0$};
\draw[fill=black] (0.53,1.61) circle (1pt);
\draw[color=black] (0.9,1.6) node {$Q_2$};
\draw[color=black] (1.15,0.35) node {$\theta$};
\draw[color=black] (0.2,1.2) node {$\gamma$};
\draw[color=black] (0.5,0.6) node {$\delta$};
\draw[color=black] (-2,2) node {$\partial C$};
\draw[fill=green,fill opacity=0.05, line width=1pt] (15,12.45) circle (19.5);
\draw (0,6.12) circle (2);
\end{scriptsize}
\end{axis}
\end{tikzpicture}
\caption{Proof of Proposition \ref{stimadeltadasotto}.}%
\label{Figure3}%
\end{figure}
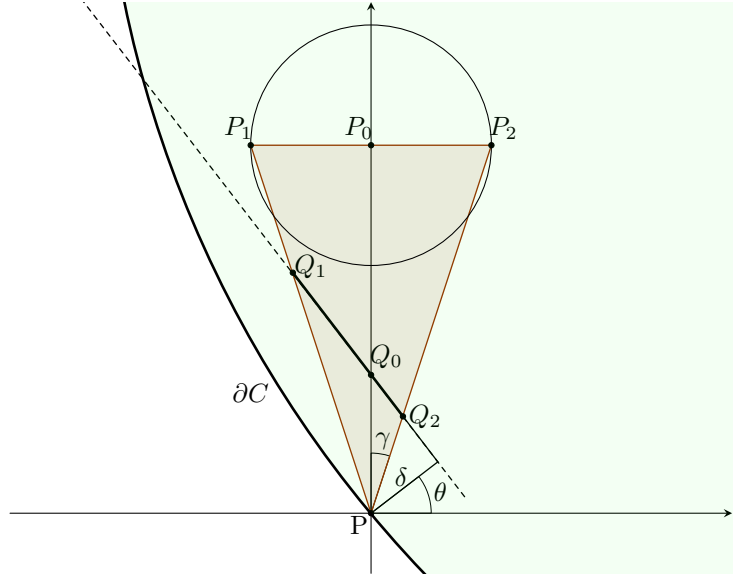
\end{proof}

To prove Theorem \ref{Theorem B-a} and Theorem \ref{irredistr} we first need a
mild variant of a classical result of J.W.S. Cassels and H. Montgomery (see
\cite{montgomery}). The following proof has been inspired by Siegel's analytic
proof of Minkowski's convex body theorem (see \cite{siegel}).

\begin{lemma}
\label{lemmaCassels}Let $U$ be a neighborhood of the origin. Then there exists
a positive constant $c$ such that for every convex symmetric body $\Omega$ in
$\mathbb{R}^{2}$ and every finite set $\left\{  u(j)\right\}  _{j=1}%
^{N}\subset\mathbb{T}^{2}$ we have%
\[
\sum_{m\in\left(  \Omega\setminus U\right)  \cap\mathbb{Z}^{2}}\left\vert
\sum_{j=1}^{N}e^{2\pi im\cdot u(j)}\right\vert ^{2}\geqslant
N\operatorname{area}({\Omega})/4-cN^{2}.
\]

\end{lemma}

\begin{proof}
Since%
\begin{align*}
\int_{\mathbb{T}^{2}}\operatorname{card}((\frac{1}{2}{\Omega}-x)\cap
\mathbb{Z}^{2})dx  &  =\int_{\mathbb{T}^{2}}\sum_{k\in\mathbb{Z}^{2}}%
\chi_{\frac{1}{2}\Omega}\left(  x+k\right)  )dx\\
&  =\int_{\mathbb{R}^{2}}\chi_{\frac{1}{2}\Omega}\left(  x\right)
dx=\operatorname{area}({\Omega})/4
\end{align*}
we can find $\overline{x}\in\left[  -\frac{1}{2},\frac{1}{2}\right)  ^{2}$
such that $\operatorname{card}((\frac{1}{2}${$\Omega$}$-\overline{x}%
)\cap\mathbb{Z}^{2})\geqslant\operatorname{area}(${$\Omega$}$)/4$. Let%
\begin{align*}
T(x)  &  =\frac{1}{\operatorname{card}((\frac{1}{2}{\Omega}-\overline{x}%
)\cap\mathbb{Z}^{2})}\left\vert \sum_{m\in\left(  \frac{1}{2}\Omega
-\overline{x}\right)  \cap\mathbb{Z}^{2}}e^{2\pi im\cdot x}\right\vert ^{2}\\
&  =\frac{1}{\operatorname{card}((\frac{1}{2}{\Omega}-\overline{x}%
)\cap\mathbb{Z}^{2})}\sum_{m,k\in\left(  \frac{1}{2}\Omega-\overline
{x}\right)  \cap\mathbb{Z}^{2}}e^{2\pi i(m-k)\cdot x}.
\end{align*}
Clearly $T$ is a non-negative trigonometric polynomial $T$. Observe that
$\widehat{T}$ is non-negative and that the support of $\widehat{T}$ is
contained in {$\Omega$} since $m,k\in\left(  \frac{1}{2}\Omega-\overline
{x}\right)  \cap\mathbb{Z}^{2}$ yields $m-k=\Omega$. Also observe that
\[
T(0)=\operatorname{card}((\frac{1}{2}{\Omega}-\overline{x})\cap\mathbb{Z}%
^{2})\geqslant\operatorname{area}({\Omega})/4.
\]
Since%
\[
0\leqslant\widehat{T}\left(  m\right)  \leqslant\widehat{T}\left(  0\right)
=\int_{\mathbb{R}^{2}}T\left(  x\right)  dx=1,
\]
it follows that
\begin{align*}
\sum_{m\in\Omega\cap\mathbb{Z}^{2}}\left\vert \sum_{j=1}^{N}e^{2\pi im\cdot
u(j)}\right\vert ^{2}  &  \geqslant\sum_{m\in\Omega\cap\mathbb{Z}^{2}}%
\widehat{T}(m)\left\vert \sum_{j=1}^{N}e^{2\pi im\cdot u(j)}\right\vert ^{2}\\
&  =\sum_{j=1}^{N}\sum_{k=1}^{N}\sum_{m\in\Omega\cap\mathbb{Z}^{2}}\widehat
{T}(m)e^{2\pi im\cdot(u(j)-u(k))}\\
&  =\sum_{j=1}^{N}\sum_{k=1}^{N}T(u(j)-u(k))\\
&  \geq NT(0)\geq N\operatorname{area}({\Omega})/4.
\end{align*}
Finally,%
\begin{align*}
\sum_{m\in\left(  \Omega\setminus U\right)  \cap\mathbb{Z}^{2}}\left\vert
\sum_{j=1}^{N}e^{2\pi im\cdot u(j)}\right\vert ^{2}  &  =\sum_{m\in\Omega
\cap\mathbb{Z}^{2}}\left\vert \sum_{j=1}^{N}e^{2\pi im\cdot u(j)}\right\vert
^{2}-\sum_{m\in U\cap\mathbb{Z}^{2}}\left\vert \sum_{j=1}^{N}e^{2\pi im\cdot
u(j)}\right\vert ^{2}\\
&  \geqslant N\operatorname{area}({\Omega})/4-\sum_{m\in U\cap\mathbb{Z}^{2}%
}N^{2}\geqslant N\operatorname{area}({\Omega})/4-cN^{2}.
\end{align*}

\end{proof}

\begin{proof}
[Proof of Theorem \ref{Theorem B-a}]Without loss of generality we can assume
$I=\left(  -\alpha,\alpha\right)  $. Since%
\[
\int_{1/2}^{1}\left\vert \widehat{\chi}_{C}\left(  \tau\rho\Theta\right)
\right\vert ^{2}\ d\tau=\int_{1/2}^{1}\left\vert \widehat{\chi}_{C}\left(
-\tau\rho\Theta\right)  \right\vert ^{2}\ d\tau,
\]
Theorem \ref{StimaDaSotto} yields, for $\rho$ large enough,%
\begin{align}
&  \int_{1/2}^{1}\left\vert \widehat{\chi}_{C}\left(  \tau\rho\left(
\cos\theta,\sin\theta\right)  \right)  \right\vert ^{2}\ d\tau\nonumber\\
&  \geq\left\{
\begin{array}
[c]{lll}%
c_{0}\rho^{-3} &  & \text{if }-\alpha<\theta<\alpha\text{ or }-\alpha
<\pi-\theta<\alpha,\\
c_{0}\rho^{-2-2\sigma} &  & \text{otherwise.}%
\end{array}
\right.  \label{c0}%
\end{align}
Let $N$ be a large positive integer. For a positive constant $\kappa$ to be
chosen later we consider the following geometric construction. Let $R_{0}$ be
the rectangle having vertices $\left(  \pm X/2,\pm Y/2\right)  $ satisfying
\[
XY=\kappa N,\hspace{1cm}X\gg Y.
\]
Let%
\[
\psi=Y/X,\hspace{1cm}M=\left[  \frac{\alpha}{\psi}\right]  =\left[
\alpha\frac{X}{Y}\right]  \ ,
\]
(here $\left[  x\right]  $ denotes the integer part of $x$). For every
$-M\leqslant j\leqslant M$ we consider the rotated rectangles $R_{j}%
:=r_{j\psi}R_{0}$, where $r_{\beta}$ is the rotation by angle $\beta$ about
the origin. See Figure \ref{chord9}. For every $m=\left(  m_{1},m_{2}\right)
\in\mathbb{Z}^{2}$ let%
\[
\Phi\left(  m\right)  =\sum_{j=-M}^{M}\chi_{R_{j}}\left(  m\right)
\]
\begin{figure}[h]
\begin{center}
\includegraphics[width=95mm]{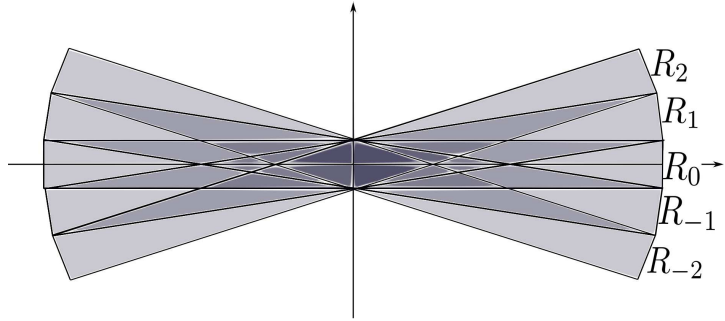}
\end{center}
\caption{A nonnegative linear combination of characteristic functions of
rotated rectangles gives a function which is smaller than (the average of)
$\widehat{\chi}_{C}$.}%
\label{chord9}%
\end{figure}We want to find a constant $\Gamma=\Gamma\left(  X\right)  $ such
that, for every $m\in\mathbb{Z}^{2}\setminus\left\{  0\right\}  $,%
\begin{equation}
\Gamma\Phi\left(  m\right)  \leqslant\left\{
\begin{array}
[c]{lll}%
c_{0}\left\vert m\right\vert ^{-3} &  & \text{if }-\alpha<\arg\left(  \pm
m\right)  <\alpha\text{,}\\
c_{0}\left\vert m\right\vert ^{-2-2\sigma} &  & \text{otherwise.}%
\end{array}
\right.  \label{dash dash}%
\end{equation}
(here $c_{0}$ is the same constant that appears in (\ref{c0})). First we
assume $m\in\cup R_{j}$ and $\left\vert m\right\vert \geqslant Y$. In this
case we have%
\begin{align*}
\Phi\left(  m\right)   &  =\sum_{j=-M}^{M}\chi_{R_{j}}\left(  m\right)
=\mathrm{\operatorname{card}}\left\{  j\in\mathbb{Z}:\ -M\leqslant j\leqslant
M\ ,\ \ m\in r_{j\psi}R_{0}\right\} \\
&  =\mathrm{\operatorname{card}}\left\{  j\in\mathbb{Z}:-M\leqslant j\leqslant
M\ ,\ r_{-j\psi}m\in R_{0}\right\}  \leqslant c\frac{X}{\left\vert
m\right\vert }\ .
\end{align*}
Indeed, the points $r_{-j\psi}m$ belong to a circle $K_{m}$ of radius
$\left\vert m\right\vert $, the lenght of $K_{m}\cap R_{0}$ is $\approx Y$ and
these points are spaced by $\psi\left\vert m\right\vert =\frac{Y}{X}\left\vert
m\right\vert $. Then the inequality%
\[
\Gamma\Phi\left(  m\right)  \leqslant c_{0}\left\vert m\right\vert ^{-3}%
\]
is a consequence of%
\[
\Gamma\frac{X}{\left\vert m\right\vert }\leqslant c\left\vert m\right\vert
^{-3},
\]
that is%
\begin{equation}
\Gamma\leqslant\inf_{m\in\cup R_{j}\text{, }\left\vert m\right\vert \geqslant
Y}\frac{c}{X\left\vert m\right\vert ^{2}}\leqslant\frac{c}{X^{3}}\ .
\label{prima dis}%
\end{equation}
We now assume $0<\left\vert m\right\vert \leq Y$, regardless of $\arg\left(
\pm m\right)  $. Then%
\[
\Phi\left(  m\right)  \leqslant2M+1=2\left[  \frac{\alpha X}{Y}\right]
+1\leqslant c\frac{\alpha X}{Y}\ .
\]
Since we want%
\[
\Gamma\Phi\left(  m\right)  \leqslant c\left\vert m\right\vert ^{-2-2\sigma},
\]
it suffices%
\[
\Gamma\frac{\alpha X}{Y}\leqslant c\left\vert m\right\vert ^{-2-2\sigma}\ .
\]
Therefore we need
\begin{equation}
\Gamma\leqslant\inf_{0<\left\vert m\right\vert \leq Y}\frac{cY}{\alpha
X\left\vert m\right\vert ^{2+2\sigma}}=\frac{c}{\alpha XY^{1+2\sigma}}\ .
\label{seconda dis}%
\end{equation}
Then by (\ref{prima dis}) and (\ref{seconda dis}) we set
\[
\Gamma=c\min\left(  \frac{1}{X^{3}},\frac{1}{\alpha XY^{1+2\sigma}}\right)
\ .
\]
We choose $\frac{1}{X^{3}}=\frac{1}{\alpha XY^{1+2\sigma}}$ and since
$XY=\kappa N$ we obtain%
\begin{align*}
X  &  =c\left(  \alpha,\kappa\right)  N^{\frac{2\sigma+1}{2\sigma+3}},\\
Y  &  =c\left(  \alpha,\kappa\right)  N^{\frac{2}{2\sigma+3}}.
\end{align*}
Then%
\[
\Gamma=c\left(  \alpha,\kappa\right)  N^{-3\left(  2\sigma+1\right)  /\left(
2\sigma+3\right)  }%
\]
yields (\ref{dash dash}). Our construction guarantees that for $\left\vert
m\right\vert \geqslant c_{1}$, we have%
\[
\int_{1/2}^{1}\left\vert \widehat{\chi}_{\tau C}\left(  m\right)  \right\vert
^{2}d\tau\geqslant\Gamma\Phi\left(  m\right)  =\Gamma\sum_{j=-M}^{M}%
\chi_{R_{j}}\left(  m\right)  .
\]
We recall that the periodic function%
\[
t\mapsto\operatorname*{card}\left(  \mathcal{P}_{N}\cap\left(  \tau
C+t\right)  \right)  -\tau^{2}N\left\vert C\right\vert
\]
has Fourier series%
\[
\sum_{m\neq0}\left(  \sum_{j=1}^{N}e^{2\pi im\cdot u\left(  j\right)
}\right)  \widehat{\chi}_{\tau C}\left(  m\right)  e^{2\pi im\cdot t}%
\]
(see \cite[p. 205]{BGT}). Then, by Parseval theorem and Lemma
\ref{lemmaCassels}, we obtain%
\begin{align*}
&  \int_{1/2}^{1}\int_{\mathbb{T}^{2}}\left\vert \operatorname*{card}\left(
\mathcal{P}_{N}\cap\left(  \tau C+t\right)  \right)  -\tau^{2}N\left\vert
C\right\vert \right\vert ^{2}dtd\tau\\
&  =\sum_{m\neq0}\left\vert \sum_{j=1}^{N}e^{2\pi im\cdot u\left(  j\right)
}\right\vert ^{2}\int_{1/2}^{1}\left\vert \widehat{\chi}_{\tau C}\left(
m\right)  \right\vert ^{2}d\tau\\
&  \geqslant\sum_{\left\vert m\right\vert \geqslant c_{1}}\left\vert
\sum_{j=1}^{N}e^{2\pi im\cdot u\left(  j\right)  }\right\vert ^{2}\Gamma
\sum_{j=-M}^{M}\chi_{R_{j}}\left(  m\right) \\
&  =\Gamma\sum_{j=-M}^{M}\sum_{\substack{\left\vert m\right\vert \geqslant
c_{1}\\m\in R_{j}}}\left\vert \sum_{j=1}^{N}e^{2\pi im\cdot u\left(  j\right)
}\right\vert ^{2}\geqslant\Gamma\sum_{j=-M}^{M}\left(  \left(
N\operatorname{area}({R}_{j})/4-cN^{2}\right)  \right) \\
&  \geqslant\Gamma\left(  2M+1\right)  \left(  \kappa N^{2}-cN^{2}\right)  .
\end{align*}
Choosing $\kappa$ large enough gives%
\begin{align*}
&  \int_{1/2}^{1}\int_{\mathbb{T}^{2}}\left\vert \operatorname*{card}\left(
\mathcal{P\cap}\left(  \tau C+t\right)  \right)  -\tau^{2}N\left\vert
C\right\vert \right\vert ^{2}dtd\tau\\
&  \geqslant c\left(  \alpha,\kappa\right)  N^{-3\left(  2\sigma+1\right)
/\left(  2\sigma+3\right)  }N^{\frac{2\sigma-1}{2\sigma+3}}N^{2}\\
&  =c\left(  \alpha,\kappa\right)  N^{\frac{2}{2\sigma+3}}.
\end{align*}

\end{proof}

\bigskip

We now begin the proof of Theorem \ref{Theorem B-b}. Let $\frac{1}{2}%
\leqslant\sigma<1$ (the case $\sigma=1$ will be addressed later) and let
$C_{\sigma}$ be a convex planar body, symmetric about the axes, such that in a
neighborhood of the point $\left(  0,-1\right)  $ the boundary $\partial
C_{\sigma}$ coincides with the graph of the function $y=\left\vert
x\right\vert ^{1/\sigma}-1$ (say for $-\varepsilon\leqslant x\leqslant
\varepsilon$). We also assume that $\partial C_{\sigma}$ is $\mathcal{C}^{2}$
and has positive curvature except at the points $\left(  0,-1\right)  $ and
$\left(  0,1\right)  $.

In the next proposition we estimate the lengths of the chords of $C_{\sigma}$.
Using the symmetry of $C_{\sigma}$ we can restrict the directions of the
inward unit normals $\Theta=\left(  \cos\theta,\sin\theta\right)  $ to the
interval $\pi/2\leqslant\theta\leqslant\pi$.

\begin{proposition}
\label{Lemma C sigma}Let $C_{\sigma}$ as above. Let $\Theta=\left(  \cos
\theta,\sin\theta\right)  $ with $\pi/2\leqslant\theta\leqslant\pi$. Then
there exists $\delta_{0},c>0$ such that, for $0<\delta\leqslant\delta_{0}$,
the chords of $C_{\sigma}$ satisfy%
\begin{equation}
\left\vert \gamma_{\Theta}\left(  \delta\right)  \right\vert \approx\left\{
\begin{array}
[c]{ll}%
\delta^{\sigma} & \text{for }0\leqslant\theta-\frac{\pi}{2}<c\,\delta
^{1-\sigma},\\
\delta^{1/2}\left(  \theta-\frac{\pi}{2}\right)  ^{\frac{2\sigma-1}{2\left(
1-\sigma\right)  }} & \text{for }c\,\delta^{1-\sigma}<\theta-\frac{\pi}{2}.
\end{array}
\right.  \label{corda C_gamma}%
\end{equation}

\end{proposition}

Observe that $\tan\left(  \theta-\frac{\pi}{2}\right)  $ is the slope of the
tangent line at the point of $\partial C_{\sigma}$ where the inward unit
normal is $\Theta$.

The notation $A\left(  x\right)  \approx B(x)$ means that there exist
constants $c_{1}$and $c_{2}$ depending on $\sigma$, such that%
\[
c_{1}A\left(  x\right)  \leqslant B\left(  x\right)  \leqslant c_{2}A\left(
x\right)  .
\]

\begin{proof}
Let us fix $c_{2}>0$ small. Let $\delta_{0}$ be small enough and $\theta
-\frac{\pi}{2}>c_{2}$. Then to the chord $\gamma_{\Theta}\left(
\delta\right)  $ is associated a small arc in $\partial C_{\sigma}$ that does
not contain the origin. Since the curvature in this arc is away from $0$ we
have $\left\vert \gamma_{\Theta}\left(  \delta\right)  \right\vert
\approx\delta^{1/2}$.

Let now $\theta-\frac{\pi}{2}\leqslant c_{2}$. If we assume $c_{2}$ and
$\delta_{0}$ small enough, then the arc associated to the chord $\gamma
_{\Theta}\left(  \delta\right)  $ is in the part of $\partial C_{\sigma}$ that
coincides with the graph of $y=\left\vert x\right\vert ^{1/\sigma}-1$.

Let $\left(  x_{0},x_{0}^{1/\sigma}-1\right)  $ the point of $\partial
C_{\sigma}$ where the tangent has slope $\tan\left(  \theta-\frac{\pi}%
{2}\right)  $. Then $\gamma_{\Theta}\left(  \delta\right)  $ concides with the
intersection of $C_{\sigma}$ with the line
\[
y=x_{0}^{1/\sigma}-1+\frac{1}{\sigma}x_{0}^{-1+1/\sigma}\left(  x-x_{0}%
\right)  +\frac{\delta}{\sin\left(  \theta\right)  }.
\]
In the following lemmas we will estimate $\left\vert \gamma_{\Theta}\left(
\delta\right)  \right\vert $ by showing that the two solutions $x_{1},x_{2}$
of the equation
\[
\left\vert x\right\vert ^{1/\sigma}-x_{0}^{1/\sigma}-\frac{1}{\sigma}%
x_{0}^{-1+1/\sigma}\left(  x-x_{0}\right)  =\frac{\delta}{\sin\left(
\theta\right)  }%
\]
satisfy%
\[
x_{2}-x_{1}\approx\left\{
\begin{array}
[c]{ll}%
\delta^{1/2}x_{0}^{\frac{2\sigma-1}{2\sigma}} & \text{for }0<\delta
<c\,x_{0}^{1/\sigma},\\
\delta^{\sigma} & \text{for }c\,x_{0}^{1/\sigma}<\delta
\end{array}
\right.
\]
(observe that $\sin\left(  \theta\right)  \approx1$ so that we can replace
$\frac{\delta}{\sin\left(  \theta\right)  }$ with $\delta$). Since
$\tan\left(  \theta-\frac{\pi}{2}\right)  =\frac{1}{\sigma}x_{0}^{-1+1/\sigma
}$, this implies (\ref{corda C_gamma}).
\end{proof}

\begin{lemma}
\label{Lemma 1}Let $1/2<\sigma<1$ and for $x\geqslant-1$ let $g\left(
x\right)  =\left(  1+x\right)  ^{1/\sigma}-1-\frac{1}{\sigma}x$. Then%
\[
g\left(  x\right)  \approx\left\{
\begin{array}
[c]{ll}%
x^{2} & \text{for }\left\vert x\right\vert \leqslant1,\\
x^{1/\sigma} & \text{for }x>1.
\end{array}
\right.
\]

\end{lemma}

\begin{proof}
Using the integral form of reminder in Taylor's formula we can write%
\[
g\left(  x\right)  =\left(  1+x\right)  ^{1/\sigma}-1-\frac{1}{\sigma}%
x=\frac{1}{\sigma}\left(  \frac{1}{\sigma}-1\right)  \int_{0}^{x}\left(
1+t\right)  ^{-2+1/\sigma}\left(  x-t\right)  dt.
\]
Let $0\leqslant x\leqslant1$, then%
\[
\int_{0}^{x}\left(  1+t\right)  ^{-2+1/\sigma}\left(  x-t\right)
dt\approx\int_{0}^{x}\left(  x-t\right)  dt=\frac{1}{2}x^{2}.
\]
Similarly, if $-\frac{1}{2}\leqslant x\leqslant0$, we have%
\[
\int_{0}^{x}\left(  1+t\right)  ^{-2+1/\sigma}\left(  x-t\right)
dt\approx\int_{x}^{0}\left(  t-x\right)  dt=\frac{1}{2}x^{2}.
\]
The estimate
\[
g\left(  x\right)  \approx x^{2}%
\]
for $-1\leqslant x<-\frac{1}{2}$ is trivial since $g\left(  x\right)  $ is
positive and bounded away from $0$ in this interval. Let now $x>1$. Since
$-1+\frac{1}{\sigma}>0$ we have%
\begin{align*}
\int_{0}^{x/2}\left(  1+t\right)  ^{-2+1/\sigma}\left(  x-t\right)  dt  &
\approx x\int_{0}^{x/2}\left(  1+t\right)  ^{-2+1/\sigma}dt=x\left[
\frac{\left(  1+t\right)  ^{-1+1/\sigma}}{-1+1/\sigma}\right]  _{0}^{x/2}\\
&  =\frac{x}{-1+1/\sigma}\left[  \left(  1+\frac{x}{2}\right)  ^{-1+1/\sigma
}-1\right]  \approx x^{1/\sigma}%
\end{align*}
and%
\[
\int_{x/2}^{x}\left(  1+t\right)  ^{-2+1/\sigma}\left(  x-t\right)  dt\approx
x^{-2+1/\sigma}\int_{x/2}^{x}\left(  x-t\right)  dt=\frac{1}{8}x^{2}%
x^{-2+1/\sigma}\approx x^{1/\sigma}%
\]
we obtain%
\[
\int_{0}^{x}\left(  1+t\right)  ^{-2+1/\sigma}\left(  x-t\right)  dx\approx
x^{1/\sigma}.
\]

\end{proof}

\begin{lemma}
\label{Lemma g}Let $1/2<\sigma<1$ and for $x\geqslant-1$ let $g\left(
x\right)  =\left(  1+x\right)  ^{1/\sigma}-1-\frac{1}{\sigma}x$. For every
$y>0$ the equation%
\[
g\left(  x\right)  =y
\]
has at most two solutions. One of them can be negative. If $\overline{x}$ is a
solution then%
\[
\left\vert \overline{x}\right\vert \approx\left\{
\begin{array}
[c]{ll}%
y^{1/2} & \text{for }0\leqslant y\leqslant1,\\
y^{\sigma} & \text{for }y>1.
\end{array}
\right.
\]

\end{lemma}

\begin{proof}
Clearly it is enough to show that $\left\vert \overline{x}\right\vert \approx
y^{1/2}$ for $\left\vert y\right\vert $ small and $\left\vert \overline
{x}\right\vert \approx y^{\sigma}$ for $y$ large. By the previous lemma there
exist constants $c_{1},c_{2}>0$ such that%
\begin{equation}
c_{1}x^{2}\leqslant g\left(  x\right)  \leqslant c_{2}x^{2}~~~\text{for
}\left\vert x\right\vert \leqslant1 \label{<1}%
\end{equation}
and%
\begin{equation}
c_{1}x^{1/\sigma}\leqslant g\left(  x\right)  \leqslant c_{2}x^{1/\sigma
}~~~\text{for }x>1. \label{>1}%
\end{equation}
Let $y<c_{1}$. Since $g\left(  \overline{x}\right)  =y$ by (\ref{>1}) we
cannot have $\left\vert \overline{x}\right\vert \geqslant1$. Hence by
(\ref{<1}) we have%
\[
c_{1}\overline{x}^{2}\leqslant y\leqslant c_{2}\overline{x}^{2}%
\]
and therefore $y\approx\overline{x}^{1/2}$. Let now $y>c_{2}$. By (\ref{<1})
we cannot have $\left\vert \overline{x}\right\vert <1$. Hence%
\[
c_{1}\overline{x}^{1/\sigma}\leqslant y\leqslant c_{2}\overline{x}^{1/\sigma}%
\]
and therefore $y\approx\overline{x}^{\sigma}$.
\end{proof}

\begin{lemma}
Let $1/2<\sigma<1$, let $x_{0}>0$ and, for every $x\in\mathbb{R}$, let%
\[
f\left(  x\right)  =\left\vert x\right\vert ^{1/\sigma}-\frac{1}{\sigma}%
x_{0}^{-1+1/\sigma}\left(  x-x_{0}\right)  -x_{0}^{1/\sigma}.
\]
Let $y>0$ and observe that the equation
\[
f\left(  x\right)  =y
\]
has one solution $x_{2}>x_{0}$ and one solution $x_{1}<x_{0}$. Then%
\[
x_{2}-x_{1}\approx\left\{
\begin{array}
[c]{ll}%
y^{1/2}x_{0}^{\frac{2\sigma-1}{2\sigma}} & \text{for }0<y<x_{0}^{1/\sigma},\\
y^{\sigma} & \text{for }x_{0}^{1/\sigma}<y.
\end{array}
\right.
\]

\end{lemma}

\begin{proof}
Let $\overline{x}>0$ such that $f\left(  \overline{x}\right)  =y$ and let%
\[
g\left(  x\right)  =\left(  1+x\right)  ^{1/\sigma}-1-\frac{1}{\sigma}x.
\]
Since for $x>0$%
\begin{align*}
f\left(  x\right)   &  =x^{1/\sigma}-\frac{1}{\sigma}x_{0}^{-1+1/\sigma
}\left(  x-x_{0}\right)  -x_{0}^{1/\sigma}\\
&  =x_{0}^{1/\sigma}\left[  \left(  1+\frac{x-x_{0}}{x_{0}}\right)
^{1/\sigma}-\frac{1}{\sigma}\left(  \frac{x-x_{0}}{x_{0}}\right)  -1\right] \\
&  =x_{0}^{1/\sigma}g\left(  \frac{x-x_{0}}{x_{0}}\right)
\end{align*}
we have%
\[
g\left(  \frac{\overline{x}-x_{0}}{x_{0}}\right)  =yx_{0}^{-1/\sigma}.
\]
By Lemma \ref{Lemma g} we have%
\[
\left\vert \frac{\overline{x}-x_{0}}{x_{0}}\right\vert \approx\left\{
\begin{array}
[c]{cl}%
y^{1/2}x_{0}^{-1/\left(  2\sigma\right)  } & \text{for }0<y<x_{0}^{1/\sigma
},\\
y^{\sigma}x_{0}^{-1} & \text{for }y>x_{0}^{1/\sigma}.
\end{array}
\right.
\]
If both $x_{1}$ and $x_{2}$ are positive, from the above estimate we easily
obtain%
\begin{equation}
x_{2}-x_{1}\approx\left\{
\begin{array}
[c]{ll}%
y^{1/2}x_{0}^{1-1/\left(  2\sigma\right)  } & \text{for }0<y<x_{0}^{1/\sigma
},\\
y^{\sigma} & \text{for }y>x_{0}^{1/\sigma}.
\end{array}
\right.  \label{x2-x1}%
\end{equation}
Assume now $x_{1}\leqslant0$ and $x_{2}>0$. Then $f\left(  0\right)  \leqslant
y=f\left(  x_{2}\right)  $. This implies
\[
x_{0}<\sigma^{\sigma/\left(  1-\sigma\right)  }x_{2}%
\]
with $\sigma^{\sigma/\left(  1-\sigma\right)  }<1$. Then $x_{2}-x_{0}\approx
x_{2}$. Let us show that $\left\vert x_{1}\right\vert <x_{2}$. This means that%
\[
f\left(  -x_{2}\right)  >f\left(  x_{2}\right)  =y.
\]
Indeed,%
\[
\left\vert -x_{2}\right\vert ^{1/\sigma}-\frac{1}{\sigma}x_{0}^{-1+1/\sigma
}\left(  -x_{2}-x_{0}\right)  -x_{0}^{1/\sigma}>\left\vert x_{2}\right\vert
^{1/\sigma}-\frac{1}{\sigma}x_{0}^{-1+1/\sigma}\left(  x_{2}-x_{0}\right)
-x_{0}^{1/\sigma}=y.
\]
This gives (\ref{x2-x1}) also in this case.
\end{proof}

The proof of Proposition \ref{Lemma C sigma} is now complete.

\begin{proposition}
Let $\frac{1}{2}\leqslant\sigma<1$. Then there exists $c>0$ such that for
every $\theta\in\left[  0,\pi\right]  $ we have%
\begin{equation}
\left\vert \widehat{\chi}_{C_{\sigma}}\left(  \pm\rho\Theta\right)
\right\vert \leqslant c\left\{
\begin{array}
[c]{ll}%
\rho^{-1-\sigma} & \text{for }\left\vert \theta-\frac{\pi}{2}\right\vert
\leqslant c\rho^{-1+\sigma},\\
\rho^{-3/2}\left\vert \theta-\frac{\pi}{2}\right\vert ^{\frac{2\sigma
-1}{2\left(  1-\sigma\right)  }} & \text{for }c\rho^{-1+\sigma}\leqslant
\left\vert \theta-\frac{\pi}{2}\right\vert \leqslant\frac{\pi}{2}.
\end{array}
\right.  \label{Stima Fourier C_sigma}%
\end{equation}

\end{proposition}

\begin{proof}
It is a consequence of Proposition \ref{Lemma C sigma},
(\ref{Stima chi hat corde}) and the symmetries of $C_{\sigma}$.
\end{proof}

\bigskip

We now consider the case $\sigma=1$ which requires a slightly different
construction. Let $C_{1}$ be a convex planar body, symmetric about the axes,
such that in a neighborhood of the point $\left(  0,-1\right)  $ the boundary
$\partial C_{1}$ coincides with the graph of the function $y=\frac{3}{4}%
x^{2}+\frac{1}{4}\left\vert x\right\vert -1$. We also assume that $\partial
C_{1}$ has positive curvature away from the points $\left(  0,\pm1\right)  $.
We have the following result.

\begin{lemma}
\label{Lemma C_1}The above convex body $C_{1}$ satisfies the following
estimates. There exists $c>0$ such that for every $\theta\in\left[
0,\pi\right]  $ we have%
\begin{equation}
\left\vert \widehat{\chi}_{C_{1}}\left(  \pm\rho\Theta\right)  \right\vert
\leqslant c_{1}\left\{
\begin{array}
[c]{ll}%
\rho^{-2} & \text{for }\left\vert \theta-\frac{\pi}{2}\right\vert \leqslant
c_{2},\\[0.3cm]%
\rho^{-3/2} & \text{for }c_{2}\leqslant\left\vert \theta-\frac{\pi}%
{2}\right\vert .
\end{array}
\right.  \label{Stima Fourier C_1}%
\end{equation}

\end{lemma}

\begin{proof}
In view of (\ref{Stima chi hat corde}) and the symmetries of $C_{1}$ it is
enough to estimate $\left\vert \gamma_{\Theta}\left(  \delta\right)
\right\vert $ for $\pi/2\leqslant\theta\leqslant\pi$. Observe that for
$\frac{1}{4}<\tan\left(  \theta-\frac{\pi}{2}\right)  \leqslant1$ there exists
a point $P=\left(  x_{0},y_{0}\right)  \in\partial C_{1}$, with $x_{0}>0$ such
that $\mathbf{n}\left(  P\right)  =\Theta$. If $\tan\left(  \theta-\frac{\pi
}{2}\right)  =\frac{1}{4}$ we do not have a unique normal at $\left(
0,-1\right)  $ and in this case $\Theta$ is the limit of the inward normal as
$x_{0}\rightarrow0^{+}$. Recall that
\[
\gamma_{\Theta}\left(  \delta\right)  =\left\{  x\in C_{1}:x\cdot
\mathbf{n}\left(  P\right)  =\inf_{y\in C}\left(  y\cdot\mathbf{n}\left(
P\right)  \right)  +\delta\right\}
\]
(see Definition \ref{Def Corda}). Since the curvature at $P$ is positive we
have $\left\vert \gamma_{\Theta}\left(  \delta\right)  \right\vert
\approx\left\vert \gamma_{\Theta}^{+}\left(  \delta\right)  \right\vert
\approx\delta^{1/2}$. Let now $0\leqslant\tan\left(  \theta-\frac{\pi}%
{2}\right)  <\frac{1}{4}$. A computation shows that the chord $\gamma_{\Theta
}\left(  \delta\right)  $ satisfies%
\begin{align*}
\left\vert \gamma_{\Theta}\left(  \delta\right)  \right\vert  &  \approx
\frac{\delta}{\sqrt{\left[  \frac{1}{4}-\tan\left(  \theta-\frac{\pi}%
{2}\right)  \right]  ^{2}+3\delta}}\\
&  \approx\left\{
\begin{array}
[c]{lll}%
\delta^{1/2} &  & \text{for }0\leqslant\frac{1}{4}-\tan\left(  \theta
-\frac{\pi}{2}\right)  <\delta^{1/2},\\[0.3cm]%
\frac{\delta}{\frac{1}{4}-\tan\left(  \theta-\frac{\pi}{2}\right)  } &  &
\text{for }\frac{1}{4}-\tan\left(  \theta-\frac{\pi}{2}\right)  \geqslant
\delta^{1/2}.
\end{array}
\right.
\end{align*}
In particular, we have the estimate $\left\vert \gamma_{\Theta}\left(
\delta\right)  \right\vert \leqslant c\delta^{1/2}$ for every $\theta$. It
follows that for a suitable $c_{2}<\frac{1}{4}$ we have%
\[
\left\vert \gamma_{\Theta}\left(  \delta\right)  \right\vert \leqslant
c_{1}\left\{
\begin{array}
[c]{lll}%
\delta^{1/2} &  & \text{for }\left\vert \tan\left(  \theta-\frac{\pi}%
{2}\right)  \right\vert >c_{2},\\[0.3cm]%
\delta &  & \text{for }\left\vert \tan\left(  \theta-\frac{\pi}{2}\right)
\right\vert \leqslant c_{2}.
\end{array}
\right.
\]
\bigskip
\end{proof}

\begin{proof}
[Proof of Theorem \ref{Theorem B-b}.]Assume first $\frac{1}{2}\leqslant
\sigma<1$ and let $C_{\sigma}$ be as above. For every positive integer $j$ let%
\[
N_{j}=\left[  j^{\frac{2\sigma+1}{2\sigma+3}}\right]  \left[  j^{\frac
{2}{2\sigma+3}}\right]  .
\]
Here $\left[  x\right]  $ denotes the integer part of $x$. Also let $K=\left[
j^{\frac{2\sigma+1}{2\sigma+3}}\right]  $, $L=\left[  j^{\frac{2}{2\sigma+3}%
}\right]  ,$
\[
u_{k,\ell}=\left(  \frac{k}{K},\frac{\ell}{L}\right)  ,
\]
and%
\[
\mathcal{P}_{N_{j}}=\left\{  u_{k,\ell}\right\}  _{\substack{k=0,\ldots
,K-1\\\ell=0,\ldots,L-1}}.
\]
Then%
\[
\int_{\mathbb{T}^{2}}\left\vert \operatorname*{card}\left(  \mathcal{P}%
_{N_{j}}\cap\left(  C_{\sigma}+t\right)  \right)  -N_{j}\left\vert C_{\sigma
}\right\vert \right\vert ^{2}~dt=\sum_{m\neq0}\left\vert \widehat{\chi
}_{C_{\sigma}}\left(  m\right)  \right\vert ^{2}\left\vert \sum_{k=0}%
^{K-1}\sum_{\ell=0}^{L-1}e^{2\pi im\cdot u_{j,k}}\right\vert ^{2}.
\]
Observe that%
\[
\sum_{k=0}^{K-1}\sum_{\ell=0}^{L-1}e^{2\pi im\cdot u_{k,\ell}}=\left\{
\begin{array}
[c]{ll}%
KL & \text{if }m_{1}=Kn_{1}\text{ and }m_{2}=Ln_{2}\text{ with }n_{1},n_{2}%
\in\mathbb{Z},\\
0 & \text{otherwise.}%
\end{array}
\right.
\]
Therefore%
\begin{align*}
&  \sum_{m\neq0}\left\vert \widehat{\chi}_{C_{\sigma}}\left(  m\right)
\right\vert ^{2}\left\vert \sum_{k=0}^{K-1}\sum_{\ell=0}^{L-1}e^{2\pi im\cdot
u_{j,k}}\right\vert ^{2}=\sum_{\left(  n_{1},n_{2}\right)  \neq\left(
0,0\right)  }\left\vert \widehat{\chi}_{C_{\sigma}}\left(  Kn_{1}%
,Ln_{2}\right)  \right\vert ^{2}K^{2}L^{2}\\
=  &  K^{2}L^{2}\sum_{\left(  n_{1},n_{2}\right)  \in\Gamma_{1}}\left\vert
\widehat{\chi}_{C_{\sigma}}\left(  Kn_{1},Ln_{2}\right)  \right\vert
^{2}+K^{2}L^{2}\sum_{\left(  n_{1},n_{2}\right)  \in\Gamma_{2}}\left\vert
\widehat{\chi}_{C_{\sigma}}\left(  Kn_{1},Ln_{2}\right)  \right\vert ^{2}\\
&  +K^{2}L^{2}\sum_{\left(  n_{1},n_{2}\right)  \in\Gamma_{3}}\left\vert
\widehat{\chi}_{C_{\sigma}}\left(  Kn_{1},Ln_{2}\right)  \right\vert ^{2}%
\end{align*}
where%
\begin{align*}
\Gamma_{1}  &  =\left\{  \left(  n_{1},n_{2}\right)  \in\mathbb{Z}%
^{2}\setminus\left\{  \left(  0,0\right)  \right\}  :\left\vert Kn_{1}%
\right\vert \geqslant c_{2}\left\vert Ln_{2}\right\vert \right\} \\
\Gamma_{2}  &  =\left\{  \left(  n_{1},n_{2}\right)  \in\mathbb{Z}%
^{2}\setminus\left\{  \left(  0,0\right)  \right\}  :c_{1}\left\vert
Ln_{2}\right\vert ^{\sigma}<\left\vert Kn_{1}\right\vert <c_{2}\left\vert
Ln_{2}\right\vert \right\} \\
\Gamma_{3}  &  =\left\{  \left(  n_{1},n_{2}\right)  \in\mathbb{Z}%
^{2}\setminus\left\{  \left(  0,0\right)  \right\}  :\left\vert Kn_{1}%
\right\vert \leqslant c_{1}\left\vert Ln_{2}\right\vert ^{\sigma}\right\}  .
\end{align*}
Let $\left(  n_{1},n_{2}\right)  \in\Gamma_{1}$, then
(\ref{Stima Fourier C_sigma}) yields%
\[
\left\vert \widehat{\chi}_{C_{\sigma}}\left(  Kn_{1},Ln_{2}\right)
\right\vert ^{2}\leqslant c\left\vert Kn_{1}\right\vert ^{-3}.
\]
It follows that%
\begin{align*}
&  K^{2}L^{2}\sum_{\left(  n_{1},n_{2}\right)  \in\Gamma_{1}}\left\vert
\widehat{\chi}_{C_{\sigma}}\left(  Kn_{1},Ln_{2}\right)  \right\vert
^{2}\leqslant cK^{2}L^{2}\sum_{n_{1}=1}^{+\infty}\sum_{\left\vert
n_{2}\right\vert \leqslant c\frac{K\left\vert n_{1}\right\vert }{L}}\left\vert
Kn_{1}\right\vert ^{-3}\\
=  &  cK^{-1}L^{2}\sum_{n_{1}=1}^{+\infty}\left\vert n_{1}\right\vert
^{-3}\left(  \frac{K\left\vert n_{1}\right\vert }{L}+1\right)  \leqslant
cL+cL^{2}K^{-1}\leqslant cj^{\frac{2}{2\sigma+3}}.
\end{align*}
Let $\left(  n_{1},n_{2}\right)  \in\Gamma_{2}$. Then
(\ref{Stima Fourier C_sigma}) yields%
\[
\left\vert \widehat{\chi}_{C_{\sigma}}\left(  Kn_{1},Ln_{2}\right)
\right\vert ^{2}\leqslant c\left\vert Ln_{2}\right\vert ^{-3}\left\vert
\frac{Kn_{1}}{Ln_{2}}\right\vert ^{\frac{2\sigma-1}{1-\sigma}}=c\left\vert
Ln_{2}\right\vert ^{-\frac{2-\sigma}{1-\sigma}}\left\vert Kn_{1}\right\vert
^{\frac{2\sigma-1}{1-\sigma}},
\]
because $\tan\left(  \theta-\frac{\pi}{2}\right)  =\left(  Kn_{1}\right)
/\left(  Ln_{2}\right)  $. Since $\left(  n_{1},n_{2}\right)  \neq0$ we also
have $c_{2}\frac{\left\vert Ln_{2}\right\vert }{K}\geqslant1$ and therefore
$\left\vert n_{2}\right\vert \geqslant c\frac{K}{L}$. Hence
\begin{align*}
&  K^{2}L^{2}\sum_{\left(  n_{1},n_{2}\right)  \in\Gamma_{2}}\left\vert
\widehat{\chi}_{C_{\sigma}}\left(  Kn_{1},Ln_{2}\right)  \right\vert ^{2}\\
&  \leqslant K^{2}L^{2}\sum_{n_{2}\geqslant c\frac{K}{L}}\,\,\sum
_{c\frac{\left\vert Ln_{2}\right\vert ^{\sigma}}{K}\leqslant\left\vert
n_{1}\right\vert \leqslant c\frac{\left\vert Ln_{2}\right\vert }{K}%
}c\left\vert Ln_{2}\right\vert ^{-\frac{2-\sigma}{1-\sigma}}\left\vert
Kn_{1}\right\vert ^{\frac{2\sigma-1}{1-\sigma}}\\
&  \leqslant cL^{-\frac{\sigma}{1-\sigma}}K^{\frac{1}{1-\sigma}}\sum
_{n_{2}\geqslant c\frac{K}{L}}\left\vert n_{2}\right\vert ^{-\frac{2-\sigma
}{1-\sigma}}\sum_{\left\vert n_{1}\right\vert \leqslant c\frac{\left\vert
Ln_{2}\right\vert }{K}}\left\vert n_{1}\right\vert ^{\frac{2\sigma-1}%
{1-\sigma}}\\
&  \leqslant cL^{-\frac{\sigma}{1-\sigma}}K^{\frac{1}{1-\sigma}}\sum
_{n_{2}\geqslant c\frac{K}{L}}\left\vert n_{2}\right\vert ^{-\frac{2-\sigma
}{1-\sigma}}\left(  \frac{\left\vert Ln_{2}\right\vert }{K}\right)
^{\frac{\sigma}{1-\sigma}}=cK\sum_{n_{2}\geqslant c\frac{K}{L}}c\left\vert
n_{2}\right\vert ^{-2}\\
&  \leqslant cL\leqslant cj^{\frac{2}{2\sigma+3}}.
\end{align*}
Let $\left(  n_{1},n_{2}\right)  \in\Gamma_{3}$. Then we have%
\[
K\left\vert n_{1}\right\vert \leqslant cL\left\vert n_{2}\right\vert ,
\]
so that (\ref{Stima Fourier C_sigma}) yields%
\[
\left\vert \widehat{\chi}_{C_{\sigma}}\left(  Kn_{1},Ln_{2}\right)
\right\vert ^{2}\leqslant c\left\vert Ln_{2}\right\vert ^{-2-2\sigma}.
\]
Hence%
\begin{align*}
&  K^{2}L^{2}\sum_{\left(  n_{1},n_{2}\right)  \in\Gamma_{3}}\left\vert
\widehat{\chi}_{C_{\sigma}}\left(  Kn_{1},Ln_{2}\right)  \right\vert
^{2}\leqslant c\,K^{2}L^{2}\sum_{n_{2}=1}^{+\infty}\sum_{\left\vert
n_{1}\right\vert \leqslant c\frac{\left\vert Ln_{2}\right\vert ^{\sigma}}{K}%
}\left\vert Ln_{2}\right\vert ^{-2-2\sigma}\\
&  \leqslant cK^{2}L^{-2\sigma}\sum_{n_{2}=1}^{+\infty}\left\vert
n_{2}\right\vert ^{-2-2\sigma}\left(  \frac{\left\vert Ln_{2}\right\vert
^{\sigma}}{K}+1\right) \\
&  \leqslant cK^{2}L^{-2\sigma}\left(  \frac{L^{\sigma}}{K}+1\right)
\leqslant cKL^{-\sigma}+cK^{2}L^{-2\sigma}\leqslant cj^{\frac{2}{2\sigma+3}}.
\end{align*}
Then%
\[
\int_{\mathbb{T}^{2}}\left\vert \operatorname*{card}\left(  \mathcal{P}%
_{N_{j}}\cap\left(  C_{\sigma}+t\right)  \right)  -N_{j}\left\vert C_{\sigma
}\right\vert \right\vert ^{2}~dt\leqslant c\,j^{\frac{2}{2\sigma+3}}\leqslant
c\,N_{j}^{\frac{2}{2\sigma+3}}.
\]

We still have to prove the case $\sigma=1$. Let $C_{1}$ be as in Lemma
\ref{Lemma C_1} and for every integer $j>0$ let $K=\left[  j^{3/5}\right]  $,
$L=\left[  j^{2/5}\right]  $, $N_{j}=\left[  j^{3/5}\right]  \left[
j^{2/5}\right]  $,%
\[
u_{k,\ell}=\left(  \frac{k}{K},\frac{\ell}{L}\right)  ,
\]
and
\[
\mathcal{P}_{N_{j}}=\left\{  u_{k,\ell}\right\}  _{\substack{k=0,\ldots
,K-1\\\ell=0,\ldots,L-1}}.
\]
Then, as in the case $\frac{1}{2}\leqslant\sigma<1$,%
\begin{align*}
&  \int_{\mathbb{T}^{2}}\left\vert \mathrm{card}\left(  \mathcal{P}_{N_{j}%
}\cap\left(  C_{1}+t\right)  \right)  -N_{j}\left\vert C_{1}\right\vert
\right\vert ^{2}\ dt\\
= &  K^{2}L^{2}\sum_{\left(  n_{1},n_{2}\right)  \in\Gamma_{1}}\left\vert
\widehat{\chi}_{C_{1}}\left(  Kn_{1},Ln_{2}\right)  \right\vert ^{2}%
+K^{2}L^{2}\sum_{\left(  n_{1},n_{2}\right)  \in\Gamma_{2}}\left\vert
\widehat{\chi}_{C_{1}}\left(  Kn_{1},Ln_{2}\right)  \right\vert ^{2}\ ,
\end{align*}
where%
\begin{align*}
\Gamma_{1} &  =\left\{  \left(  n_{1},n_{2}\right)  \in\mathbb{Z}^{2}%
\setminus\left\{  \left(  0,0\right)  \right\}  :\left\vert Kn_{1}\right\vert
\leqslant c\left\vert Ln_{2}\right\vert \right\}  ,\\
\Gamma_{2} &  =\left\{  \left(  n_{1},n_{2}\right)  \in\mathbb{Z}^{2}%
\setminus\left\{  \left(  0,0\right)  \right\}  :\left\vert Kn_{1}\right\vert
>c\left\vert Ln_{2}\right\vert \right\}  .
\end{align*}
For the first series we have%
\begin{align*}
&  K^{2}L^{2}\sum_{\left(  n_{1},n_{2}\right)  \in\Gamma_{1}}\left\vert
\widehat{\chi}_{C_{1}}\left(  Kn_{1},Ln_{2}\right)  \right\vert ^{2}\leqslant
cK^{2}L^{2}\sum_{n_{2}=1}^{+\infty}\sum_{\left\vert n_{1}\right\vert
<\frac{\left\vert Ln_{2}\right\vert }{K}}\left\vert Ln_{2}\right\vert ^{-4}\\
&  \leqslant cK^{2}L^{-2}\sum_{n_{2}=1}^{+\infty}\left\vert n_{2}\right\vert
^{-4}\left(  \frac{\left\vert Ln_{2}\right\vert }{K}+1\right)  \leqslant
cKL^{-1}+K^{2}L^{-2}\approx j^{2/5}\approx N_{j}^{2/5}\ .
\end{align*}
For the second series we have%
\begin{align*}
&  K^{2}L^{2}\sum_{\left(  n_{1},n_{2}\right)  \in\Gamma_{2}}\left\vert
\widehat{\chi}_{C_{1}}\left(  Kn_{1},Ln_{2}\right)  \right\vert ^{2}\leqslant
cK^{2}L^{2}\sum_{n_{1}=1}^{+\infty}\sum_{n_{2}\leqslant\frac{\left\vert
Kn_{1}\right\vert }{L}}\left\vert Kn_{1}\right\vert ^{-3}\\
&  \leqslant cK^{-1}L^{2}\sum_{n_{1}=1}^{+\infty}\left\vert n_{1}\right\vert
^{-3}\left(  \frac{\left\vert Kn_{1}\right\vert }{L}+1\right)  \leqslant
cj^{\frac{2}{5}}\approx N_{j}^{2/5}.
\end{align*}

\end{proof}

\begin{proof}
[Proof of Proposition \ref{Prop2/5}]We show that we can apply Theorem
\ref{Theorem B-a} with $\sigma=1$ to the convex $C$. By Proposition
\ref{stimadeltadasotto} we have%
\[
\left\vert \gamma_{\Theta}\left(  \delta\right)  \right\vert \geqslant c\delta
\]
for every $\Theta$. Since $\partial C$ is piecewise $\mathcal{C}^{2}$ and $C$
is not a polygon there is an arc $\Gamma$ in $\partial C$ which is
$\mathcal{C}^{2}$ and where the curvature is away from zero. For every
$P\in\Gamma$ let $\Theta$ be the inward unit normal at $P$. Then, by the
argument in Remark \ref{CordaC2} we have
\[
\left\vert \gamma_{\Theta}\left(  \delta\right)  \right\vert \geqslant
c\delta^{1/2}.
\]

\end{proof}

\begin{proof}
[Proof of Theorem \ref{irredistr}]By Theorem \ref{StimaDaSotto}, if $\xi
\in\mathbb{R}^{2}$, $\left\vert \xi\right\vert \geqslant c_{1}$ we have%
\[
\int_{1/2}^{1}\left\vert \widehat{\chi}_{C}\left(  \tau\xi\right)  \right\vert
^{2}d\tau\geqslant c_{2}\left\vert \xi\right\vert ^{-2-2\sigma}.
\]
To apply Lemma \ref{lemmaCassels}, let $\Omega=\left\{  t\in\mathbb{R}%
^{2}:\left\vert t\right\vert \leqslant c_{3}\sqrt{N}\right\}  $, where $c_{3}$
is a constant that will be chosen later and let $U=\left\{  t\in\mathbb{R}%
^{2}:\left\vert t\right\vert \leqslant c_{1}\right\}  $. Then, as in the proof
of Theorem \ref{Theorem B-a}, we have
\begin{align*}
&  \int_{1/2}^{1}\int_{\mathbb{T}^{2}}\left\vert \operatorname*{card}\left(
\mathcal{P}_{N}\cap\left(  \tau C+t\right)  \right)  -\tau^{2}N\left\vert
C\right\vert \right\vert ^{2}dtd\tau\\
&  =\int_{1/2}^{1}\sum_{m\neq0}\left\vert \sum_{j=1}^{N}e^{2\pi im\cdot
u\left(  j\right)  }\right\vert ^{2}\left\vert \widehat{\chi}_{\tau C}\left(
m\right)  \right\vert ^{2}d\tau\\
&  \geqslant\sum_{m\in\Omega\setminus U}\left\vert \sum_{j=1}^{N}e^{2\pi
im\cdot u\left(  j\right)  }\right\vert ^{2}\int_{1/2}^{1}\tau^{2}\left\vert
\widehat{\chi}_{\Omega}\left(  \tau m\right)  \right\vert ^{2}d\tau\\
&  \geqslant c_{2}\sum_{m\in\Omega\setminus U}\left\vert \sum_{j=1}^{N}e^{2\pi
im\cdot u\left(  j\right)  }\right\vert ^{2}\left\vert m\right\vert
^{-2-2\sigma}\\
&  \geqslant c_{2}\left(  c_{3}\sqrt{N}\right)  ^{-2-2\sigma}\sum_{m\in
\Omega\setminus U}\left\vert \sum_{j=1}^{N}e^{2\pi im\cdot u\left(  j\right)
}\right\vert ^{2}\\
&  \geqslant c_{2}\left(  c_{3}\sqrt{N}\right)  ^{-2-2\sigma}\left(
N\operatorname{area}({\Omega})/4-cN^{2}\right) \\
&  =c_{2}\left(  c_{3}\right)  ^{-2-2\sigma}N^{-1-\sigma}\left(  N\pi
c_{3}^{2}N/4-cN^{2}\right)  =cN^{1-\sigma}\;,
\end{align*}
provided that $\pi c_{3}^{2}/4>c$.
\end{proof}

\begin{proof}
[Proof of Theorem \ref{Equivalenza}]Assume that condition a) holds true. Since
$C$ is convex, at every point $P\in\partial C$ there exist a left and a right
tangent. Observe that if they differed, then we would have $\left\vert
\gamma_{\Theta}^{-}\left(  \delta\right)  \right\vert +\left\vert
\gamma_{\Theta}^{+}\left(  \delta\right)  \right\vert \leqslant c\delta$ which
is incompatible with (\ref{condizione corda}). This means that $\Gamma
^{\prime}\left(  s\right)  $ exists for every $s$. We denote by $\mathbf{n}%
\left(  s\right)  $ the inward unit normal at $\Gamma\left(  s\right)  $. Let
us fix $s_{1},s_{2}$. We can clearly assume that $\left\vert s_{1}%
-s_{2}\right\vert $ is small. Let $\Theta=\mathbf{n}\left(  s_{1}\right)  $
and let $\delta=\left[  \Gamma\left(  s_{2}\right)  -\Gamma\left(
s_{1}\right)  \right]  \cdot\mathbf{n}\left(  s_{1}\right)  $. Observe that
$\left\vert \left[  \Gamma\left(  s_{2}\right)  -\Gamma\left(  s_{1}\right)
\right]  \cdot\Gamma^{\prime}\left(  s\right)  \right\vert $ is the length of
$\gamma_{\Theta}^{+}\left(  \delta\right)  $ or $\gamma_{\Theta}^{-}\left(
\delta\right)  $ (according to the orientation of the curve). See Figure
\ref{chord8}. \begin{figure}[pth]
\begin{center}
\includegraphics[width=60mm]{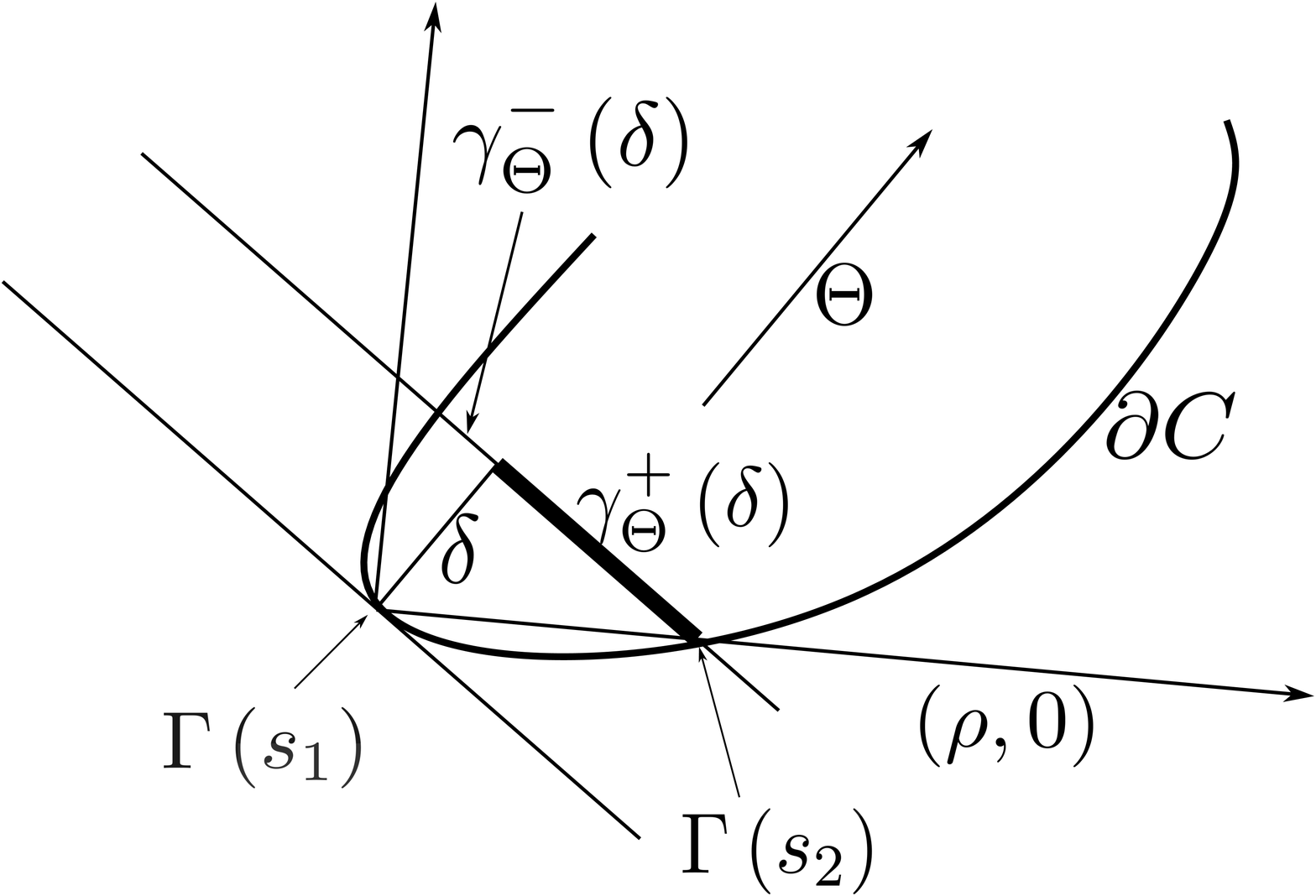}
\end{center}
\caption{Proof of Theorem \ref{Equivalenza}, first part.}%
\label{chord8}%
\end{figure}Then (\ref{condizione corda}) yields
\[
\left\vert \left[  \Gamma\left(  s_{2}\right)  -\Gamma\left(  s_{1}\right)
\right]  \cdot\Gamma^{\prime}\left(  s\right)  \right\vert \geqslant
c\delta^{1/\left(  1+\alpha\right)  },
\]
so that%
\begin{align}
0  &  \leqslant\left[  \Gamma\left(  s_{2}\right)  -\Gamma\left(
s_{1}\right)  \right]  \cdot\mathbf{n}\left(  s_{1}\right)  \leqslant
c\left\vert \left[  \Gamma\left(  s_{2}\right)  -\Gamma\left(  s_{1}\right)
\right]  \cdot\Gamma^{\prime}\left(  s_{1}\right)  \right\vert ^{1+\alpha
}\label{ns1}\\
&  \leqslant c\left\vert \Gamma\left(  s_{2}\right)  -\Gamma\left(
s_{1}\right)  \right\vert ^{1+\alpha}.\nonumber
\end{align}
Similarly%
\begin{equation}
0\leqslant\left[  \Gamma\left(  s_{1}\right)  -\Gamma\left(  s_{2}\right)
\right]  \cdot\mathbf{n}\left(  s_{2}\right)  \leqslant c\left\vert
\Gamma\left(  s_{2}\right)  -\Gamma\left(  s_{1}\right)  \right\vert
^{1+\alpha}. \label{ns2}%
\end{equation}
We claim that%
\[
\left\vert \mathbf{n}\left(  s_{1}\right)  -\mathbf{n}\left(  s_{2}\right)
\right\vert \leqslant c\left\vert s_{2}-s_{1}\right\vert ^{\alpha}.
\]
Indeed, let $\rho=\left\vert \Gamma\left(  s_{1}\right)  -\Gamma\left(
s_{2}\right)  \right\vert $ and let us choose coordinates so that%
\[
\Gamma\left(  s_{2}\right)  -\Gamma\left(  s_{1}\right)  =\left(
\rho,0\right)  .
\]
Then from (\ref{ns1}) and (\ref{ns2}) we obtain, writing $\mathbf{n}\left(
s\right)  =\left(  n_{1}\left(  s\right)  ,n_{2}\left(  s\right)  \right)  $,
that%
\begin{align*}
0  &  \leqslant n_{1}\left(  s_{1}\right)  \leqslant c\rho^{\alpha}\\
0  &  \leqslant-n_{1}\left(  s_{2}\right)  \leqslant c\rho^{\alpha}%
\end{align*}
and therefore%
\[
0\leqslant n_{1}\left(  s_{1}\right)  -n_{1}\left(  s_{2}\right)  \leqslant
c\rho^{\alpha}.
\]
Since $\left\vert \mathbf{n}\left(  s_{1}\right)  \right\vert =\left\vert
\mathbf{n}\left(  s_{2}\right)  \right\vert =1$ and $n_{2}\left(
s_{1}\right)  ,n_{2}\left(  s_{2}\right)  >0$ we have%
\begin{align*}
\left\vert n_{2}\left(  s_{1}\right)  -n_{2}\left(  s_{2}\right)  \right\vert
&  =\left\vert \sqrt{1-\left[  n_{1}\left(  s_{1}\right)  \right]  ^{2}}%
-\sqrt{1-\left[  n_{1}\left(  s_{2}\right)  \right]  ^{2}}\right\vert \\
&  \leqslant c\left\vert n_{1}\left(  s_{1}\right)  -n_{1}\left(
s_{2}\right)  \right\vert \leqslant c\rho^{\alpha}.
\end{align*}
It follows that%
\[
\left\vert \mathbf{n}\left(  s_{1}\right)  -\mathbf{n}\left(  s_{2}\right)
\right\vert \leqslant c\rho^{\alpha}\leqslant c\left\vert s_{2}-s_{1}%
\right\vert ^{\alpha},
\]
so that%
\[
\left\vert \Gamma^{\prime}\left(  s_{1}\right)  -\Gamma^{\prime}\left(
s_{2}\right)  \right\vert \leqslant M\left\vert s_{1}-s_{2}\right\vert
^{\alpha}.
\]
Assume now that b) holds true. Let us fix a direction $\Theta$ and let
$\Gamma\left(  s_{0}\right)  $ be a point on $\partial C$ where $\mathbf{n}%
\left(  s_{0}\right)  =\Theta$. For $\delta$ small enough there exist two
points $\Gamma\left(  s_{1}\right)  $ and $\Gamma\left(  s_{2}\right)  $ on
$\partial C$ such that%
\begin{align*}
\left\vert \gamma_{\Theta}^{+}\left(  \delta\right)  \right\vert  &
=\left\vert \left[  \Gamma\left(  s_{2}\right)  -\Gamma\left(  s_{0}\right)
\right]  \cdot\Gamma^{\prime}\left(  s_{0}\right)  \right\vert ,\\
\left\vert \gamma_{\Theta}^{-}\left(  \delta\right)  \right\vert  &
=\left\vert \left[  \Gamma\left(  s_{1}\right)  -\Gamma\left(  s_{0}\right)
\right]  \cdot\Gamma^{\prime}\left(  s_{0}\right)  \right\vert .
\end{align*}
and%
\[
\left[  \Gamma\left(  s_{j}\right)  -\Gamma\left(  s_{0}\right)  \right]
\cdot\mathbf{n}\left(  s_{0}\right)  =\delta,\ \ \ \ \ \ j=1,2.
\]
\begin{figure}[pth]
\begin{center}
\includegraphics[width=57mm]{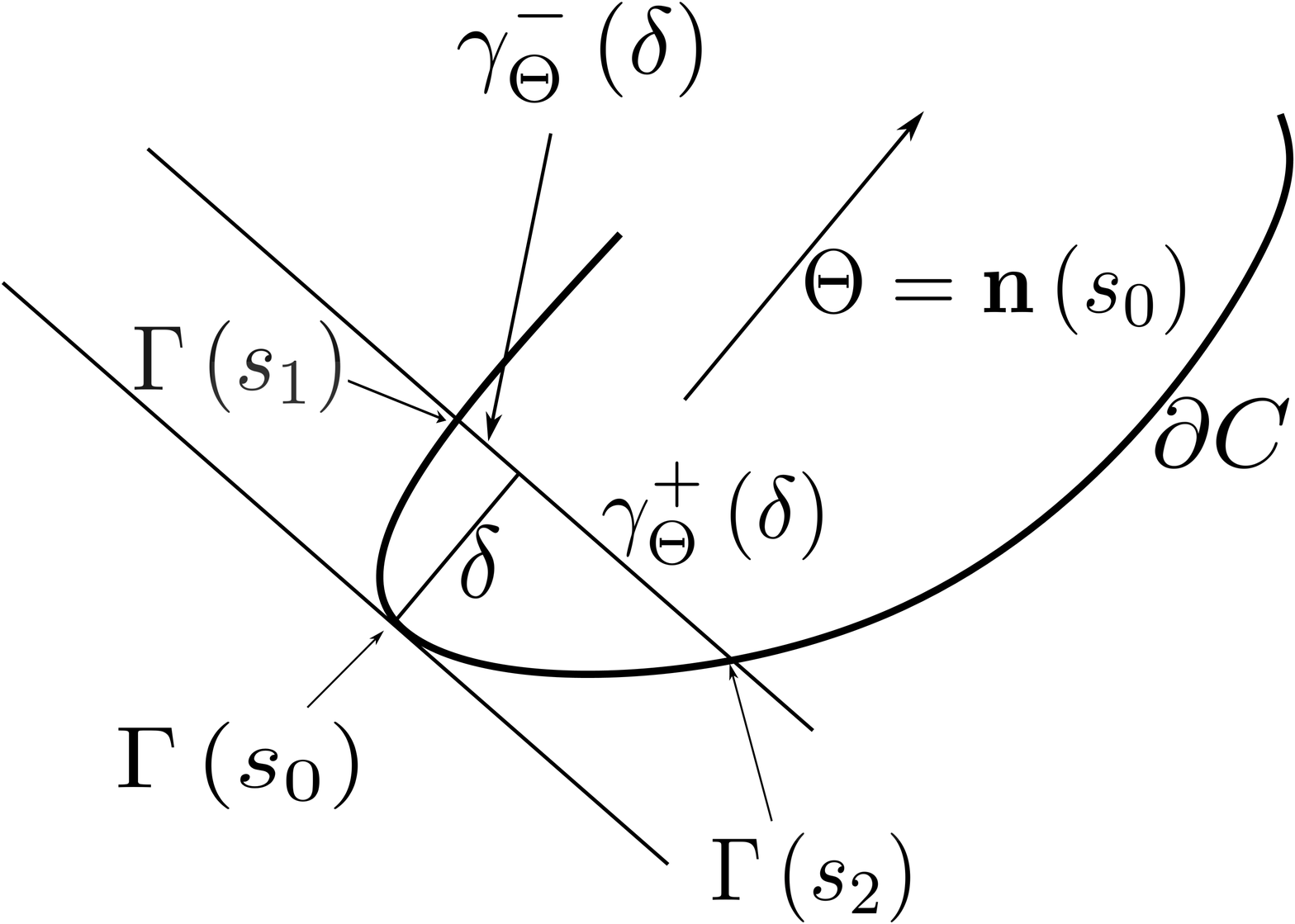}
\end{center}
\caption{Proof of Theorem \ref{Equivalenza}, second part.}%
\label{chord99}%
\end{figure}Then (see Figure \ref{chord99}), for $j=1,2$ we have%
\begin{align*}
\left[  \Gamma\left(  s_{j}\right)  -\Gamma\left(  s_{0}\right)  \right]
\cdot\mathbf{n}\left(  s_{0}\right)   &  =\left\vert \int_{s_{0}}^{s_{j}%
}\Gamma^{\prime}\left(  \tau\right)  \cdot\mathbf{n}\left(  s_{0}\right)
d\tau\right\vert \\
&  =\left\vert \int_{s_{0}}^{s_{j}}\left[  \Gamma^{\prime}\left(  \tau\right)
-\Gamma^{\prime}\left(  s_{0}\right)  \right]  \cdot\mathbf{n}\left(
s_{0}\right)  d\tau\right\vert \\
&  \leqslant\left\vert \int_{s_{0}}^{s_{j}}\left\vert \Gamma^{\prime}\left(
\tau\right)  -\Gamma^{\prime}\left(  s_{0}\right)  \right\vert d\tau
\right\vert \\
&  \leqslant\left\vert \int_{s_{0}}^{s_{j}}\left\vert \tau-s_{0}\right\vert
^{\alpha}d\tau\right\vert \leqslant c\left\vert s_{j}-s_{0}\right\vert
^{\alpha+1}.
\end{align*}
Hence%
\[
\delta^{1/\left(  \alpha+1\right)  }\leqslant c\left\vert s_{1}-s_{0}%
\right\vert \leqslant c\left\vert \Gamma\left(  s_{1}\right)  -\Gamma\left(
s_{0}\right)  \right\vert \leqslant\left\vert \gamma_{\Theta}^{-}\left(
\delta\right)  \right\vert
\]
and similarly $\delta^{1/\left(  \alpha+1\right)  }\leqslant\left\vert
\gamma_{\Theta}^{+}\left(  \delta\right)  \right\vert .$
\end{proof}

\end{document}